\documentclass[12pt]{amsart}
\usepackage{amssymb,amsmath,amsthm,mathrsfs,MnSymbol,stmaryrd, enumitem}
\usepackage{graphicx}
\usepackage[all,cmtip]{xy}
\usepackage{tikz-cd}
\usepackage{multirow,scalerel,arydshln,caption,subcaption}   
\usepackage{hyperref}
\usepackage[msc-links]{amsrefs}
\hypersetup{
  colorlinks   = true,	
  urlcolor     = blue,
  linkcolor    = blue,
  citecolor   = red
}
\calclayout
\allowdisplaybreaks
\newtheorem{theorem}{Theorem}
\newtheorem{lemma}[theorem]{Lemma}
\newtheorem{remark}[theorem]{Remark}
\newtheorem{corollary}[theorem]{Corollary}

\newtheorem{proposition}[theorem]{Proposition}
\numberwithin{theorem}{section} \numberwithin{equation}{section}
\newcommand{\beq}{\begin{small} \begin{equation}}
\newcommand{\eeq}{\end{equation} \end{small}}
\newcommand{\beqn}{\begin{small} \begin{equation*}}
\newcommand{\eeqn}{\end{equation*} \end{small}}
\newcommand\scalemath[2]{\scalebox{#1}{\mbox{\ensuremath{\displaystyle #2}}}}
\begin{document}
\title[On the duality of F-theory and the CHL string]{On the duality of F-theory and the CHL string in seven dimensions}
\author{Adrian Clingher}
\address{Department of Mathematics and Statistics, University of Missouri - St. Louis, St. Louis, MO 63121}
\email{clinghera@umsl.edu}
\author{Andreas Malmendier}
\address{Department of Mathematics \& Statistics, Utah State University, Logan, UT 84322}
\email{andreas.malmendier@usu.edu}
\begin{abstract}
We show that the duality between F-theory and the CHL string in seven dimensions defines algebraic correspondences between K3 surfaces polarized by the rank-ten lattices $H \oplus N$ and $H\oplus E_8(-2)$. In the special case when the F-theory admits an additional anti-symplectic involution or, equivalently, the CHL string admits a symplectic one, both moduli spaces coincide. In this case, we derive an explicit parametrization for the F-theory compactifications dual to the CHL string, using an auxiliary genus-one curve, based on a construction given by Andr\'e Weil. 
 \end{abstract}
\maketitle
\section{Introduction}
By standard lattice-theoretic observations \cite{MR633160}, one has the following lattice isomorphism:
\beq
\label{eqn:isom0}
 H \oplus E_8(-2) \ \cong \  H(2) \oplus N  \,.
\eeq 
Here, $E_8$ is the positive definite root lattices associated with the $E_8$ root systems, $H$ is the unique even unimodular hyperbolic rank-two lattice and $N$ is the negative definite rank-eight Nikulin lattice (see \cite{MR728142}*{Def.~5.3}, for a definition). The notation $L(\lambda)$ refers to the lattice obtained from $L$ after scaling of its bilinear form by $\lambda \in \mathbb{Z}$. Moreover, as proved in \cite{MR4069236}, the lattice isomorphism~(\ref{eqn:isom0}) implies that the lattice $H(2) \oplus  E_8(-2)$ admits two special overlattices, namely the lattice $H \oplus N$, necessarily of index four, and $H \oplus E_8(-2)$ of index two.
\par In the above contexts, we shall consider $\mathfrak{M}_{H \oplus N}$ and  $\mathfrak{M}_{H \oplus E_8(-2)}$, as moduli spaces of complex algebraic K3 surfaces with lattice polarizations of type $H \oplus N$ and  $H \oplus E_8(-2)$. Both these moduli spaces are 10-dimensional. And the K3 surfaces classified by them can be described explicitly, via Weierstrass models. The first K3 family was studied by van Geemen and Sarti \cite{MR2274533}. These K3 surfaces carry a canonical Jacobian elliptic fibration with an element of order two (or 2-torsion section) in its Mordell-Weil group, which, in turn, determines a special class of K3 involution referred to in the literature as van Geemen-Sarti involution. The K3 surfaces in the second family, associated with $H \oplus E_8(-2)$ polarizations, carry canonical Jacobian elliptic fibrations, but in this case one has compatible anti-symplectic involutions, with the property that the minimal resolution of the associated $\mathbb{Z}/2\mathbb{Z}$ 
quotient is a rational elliptic surface \cite{MR4069236}.
\par The lattice isomorphism (\ref{eqn:isom0}), via Hodge theory, implies then that the two K3 families from above are related via special algebraic correspondences. These correspondences are governed by markings of even-eight configurations in the N\'eron-Severi group $\mathrm{NS}(\mathcal{X})$ for $\mathcal{X} \in \mathfrak{M}_{H \oplus N}$ and pairs $\{ \pm \theta \} \subset \mathrm{Br}(\widetilde{\mathcal{Y}})$ of Brauer group elements for $\widetilde{\mathcal{Y}} \in \mathfrak{M}_{H \oplus D_4(-1)^{\oplus 2}}$. Here, $\widetilde{\mathcal{Y}}$ is the K3 surface obtained from the Nikulin construction using a canonical involution on $\mathcal{Y}$. These algebraic correspondences stem from classical constructions in the (mathematics) literature \cites{MR2166182, MR2818742, MR3995925,MR4069236, MR3833460}.  
\par The above mentioned construction has a remarkable consequence in string theory- it provides a mathematical framework for a class of string dualities linked to the so-called \emph{CHL string}, named after Chaudhuri, Hockney, and Lykken \cite{MR1351447}. The CHL string is obtained from the more familiar $E_8 \times E_8$ heterotic string on a torus $T^2$, as a certain $\mathbb{Z}/2\mathbb{Z}$ quotient.  The Narain construction shows that the physical CHL moduli space is identical with the moduli space $\mathfrak{M}_{H \oplus E_8(-2)}$  of K3 surfaces principally polarized by the lattice $H \oplus E_8(-2)$ \cites{MR1479699, MR1772271}. 
\par F-theory, i.e., compactifications of the type-IIB string theory in which the complex coupling varies over a base, is a powerful tool for analyzing the non-perturbative aspects of heterotic string compactifications \cites{MR1409284,MR1412112}.  The simplest F-theory constructions are K3 surfaces that admit Jacobian elliptic fibrations over $\mathbb{P}^1$. It is well known (see \cite{MR1797021}) that, for F-theory backgrounds with non-zero flux given by a $B$-field along the base curve $\mathbb{P}^1$, the value of this flux is quantized and fixed to half the K\"ahler class of $\mathbb{P}^1$. In geometric language, this structure is equivalent to a Jacobian elliptic fibration supported on the K3 surface, admitting a 2-torsion section, i.e., a van Geemen-Sarti involution. The existence of the van Geemen-Sarti involution then implies that the F-theory on the K3 surface is to be further compactified on a circle $S^1$ yielding a 7-dimensional compactification.
\par It follows from above that the K3 moduli spaces $\mathfrak{M}_{H \oplus N}$ and  $\mathfrak{M}_{H \oplus E_8(-2)}$ are the moduli space of F-theory models and of the dual of the CHL string in seven dimensions, respectively. In this article, we shall use classical geometrical notions - algebraic correspondences, even-eight configurations, and elements of the Brauer group - to give a precise mathematical framework for the F-theory/CHL string duality.
\par Particular situations of the above duality are also interesting mathematically. One may study F-theory vacua and CHL string backgrounds in the presence of additional structure.  Our framework then allows us to give explicit parameterizations for F-theory vacua and CHL string backgrounds in the presence of an additional involution.  For instance, a natural 6-dimensional subspace that is contained simultaneously in both aforementioned physical moduli spaces is the moduli space of K3 surfaces polarized by the lattice $H \oplus N_0(-1)$. Here, $N_0$ is a positive definite lattice of rank $12$, determinant $2^8$, and its Gram matrix will be computed explicitly. This is a subspace where the F-theory admits an additional anti-symplectic involution, and, on the CHL string side, one has an additional a symplectic involution. We derive an explicit parametrization for elements of the moduli space using an auxiliary genus-one curve. This parametrization is based on a construction by Andr\'e Weil \cite{MR717601}, in which the Abel-Jacobi map is used to obtain embeddings of genus-one curves as symmetric divisors of bi-degree $(2,2)$ in $\mathbb{F}_0 =\mathbb{P}^1 \times \mathbb{P}^1$. In fact, we show that the algebraic correspondences for the F-theory backgrounds and vacua of the CHL string are \emph{double-quadrics} obtained from double covers of $\mathbb{F}_0$, where the branching divisor consists of a symmetric divisor of bi-degree $(2,2)$ and an additional collection of lines.
\par The F-theory moduli space has another natural 6-dimensional subspace, namely the moduli space of K3 surfaces polarized by the lattice $\langle 2 \rangle  \oplus \langle -2 \rangle \oplus D_4(-1)^{\oplus 3}$. This special situation corresponds to the case when, on the F-theory side, surfaces admit an additional symplectic involution and, on the CHL string side,  a special anti-symplectic involution exist.  The K3 surfaces associated with the CHL string in the situation above also carry a beautiful geometric description: they are special \emph{double-sextic surfaces}, i.e., they can be obtained as minimal resolutions from double covers of the projective plane branched over a configuration of three distinct lines coincident in a point and an additional generic cubic. The latter divisor gives rise to an elliptic curve capturing part of the K3 moduli coordinates. This fact relates the current work to pervious work by the authors in \cite{Clingher:2020baq}. 

\par It is important to note that, in the two special examples described above, both involving 6-dimensional subspaces of $ \mathfrak{M}_{H \oplus N}$,  an elliptic curve naturally emerges.  And this is \emph{not} the elliptic curve from which the CHL string is constructed, but rather a Seiberg-Witten type curve that  parameterizes certain moduli of the F-theory/CHL vacua under consideration.
\par This article is structured as follows: in Section~\ref{sec:CHL} we review the physics of the duality between F-theory and the CHL string. We will then argue that the lattice polarized K3 surfaces under consideration correspond to F-theory models that are dual of the CHL string in seven dimensions. In Section~\ref{sec:constructions} we give a construction for families of lattice polarized K3 surfaces with canonical anti-symplectic and symplectic involution, respectively. In Section~\ref{sec:geometry} we prove that the duality between F-theory with discrete flux and the CHL string determines certain algebraic correspondences between pairs of K3 surfaces polarized by the rank-ten lattices $H \oplus N$ and $H\oplus E_8(-2)$, respectively. In Section~\ref{sec:specialization}, we restrict our attention to the moduli space of K3 surfaces polarized by the lattice $H \oplus N_0(-1)$, which is contained simultaneously in both $\mathfrak{M}_{H \oplus N}$ and  $\mathfrak{M}_{H \oplus E_8(-2)}$. We derive an explicit classification of these surfaces, using an auxiliary genus-one curve. In Section~\ref{sec:specialization2} we investigate a special region of the F-theory moduli space corresponding to K3 surfaces principally polarized by the lattice $\langle 2 \rangle  \oplus \langle -2 \rangle \oplus D_4(-1)^{\oplus 3}$. Some concluding remarks are included in Section~\ref{sec:summary}.
\par This article is based on several prior papers by the authors and their collaborators  \cites{MR2369941, MR2824841, CM:2018b, MR3767270, MR2935386, MR2854198, MR3366121, MR3712162,  MR3798883, MR3995925, MR3992148, MR4015343, MR4099481, CHM19, MR4160930, CMS:2020}, as well as several other works \cites{MR0429917, MR1013073, MR894512, MR1023921, MR1877757, MR1013162, MR1458752, MR2409557, MR1703212, MR2427457, MR3263663, MR2306633, MR728142, MR3563178, MR3010125, Clingher:2020baq}. 
\subsection*{Acknowledgments}
The authors would like to thank the referees for their insightful comments, in particular with regards to the correct physical interpretation of our results.  We would also like to thank Matthias Sch\"utt for help in correcting a mistake occurring in an earlier version of this article. A.C. acknowledges support from a UMSL Mid-Career Research Grant. A.M. acknowledges support from the Simons Foundation through grant no.~202367.
\section{The duality between F-theory and the CHL string}
\label{sec:CHL}
Compactifications of the type-IIB string theory in which the complex coupling varies over a base are generically referred to as F-theory \cites{MR1409284,MR1412112}. One of the simplest F-theory construction corresponds to K3 surfaces that are elliptically fibered over $\mathbb{P}^1$, in physics equivalent to a type-IIB string theory compactified on $\mathbb{P}^1$ and hence eight-dimensional in the presence of 7-branes \cite{MR3366121}. In this way, an elliptically fibered K3 surface with section and fiber $F=\mathbb{C}/(\mathbb{Z}\oplus \mathbb{Z}\tau)$ defines an F-theory vacuum in eight dimensions where the complex-valued scalar axio-dilaton field $\tau$ of the type-IIB string theory is allowed to be multi-valued and undergo monodromy transformations in $\mathrm{SL}(2,\mathbb{Z})$ when encircling defects of co-dimension one. Kodaira's table of singular fibers \cite{MR0184257} gives the precise dictionary between characteristics of the elliptic fibration and the content of the 7-branes present in the physical theory and the local monodromy of $\tau$. It is well-known that the moduli space of these F-theory models is isomorphic to the moduli space of the heterotic string compactified on a two-torus $T^2$ -- equipped with a complex structure and complexified K\"ahler modulus -- together with a principal $G$-bundle where $G$ is the gauge group of the heterotic string, i.e., $G=E_8 \times E_8 \rtimes \mathbb{Z}_2$ or $\mathrm{Spin}(32)/\mathbb{Z}_2$ or a subgroup of these \cites{MR796086, MR844696, MR800347, MR1154939}. In fact, the moduli spaces for these physical theories are given by the same Narain space which is the quotient of the symmetric space for $\mathrm{O}(2,18)$ by a particular arithmetic group \cite{MR834338}. This is the basic form of the F-theory/heterotic string duality.
\par one can also consider the \emph{CHL string} \cites{MR1615617,MR1375878, MR1621170, MR1797021}. The CHL string is obtained from an $E_8 \times E_8$-heterotic string as a $\mathbb{Z}_2$-quotient. The quotient is obtained from a specific involution, called the \emph{CHL involution} which is due to Chaudhuri and Polchinski in \cite{MR1375878}: the CHL involution $\imath_\mathrm{CHL}$ acts by a half-period shift on the elliptic curve (obtained from the two-torus $T^2$ equipped with the given complex structure), acts trivially on the complex K\"ahler modulus, and permutes the two $E_8$'s of the gauge bundle.  The moduli space for the CHL string compactified on the elliptic curve is then obtained as follows: first, the Narain construction yields the moduli space of the $E_8 \times E_8$-heterotic string compactified on $T^2$ as the double coset space
\beq
\label{eqn:moduli}
 \mathrm{O}(\mathrm{Nar}) \backslash \mathrm{O}(\mathrm{Nar} \otimes \mathbb{R})  / \mathrm{K} \,,
\eeq
where $\mathrm{Nar} =  H \oplus H \oplus E_8(-1) \oplus E_8(-1)$ is the Narain lattice and $\mathrm{K} \subset  \mathrm{O}(\mathrm{Nar} \otimes \mathbb{R})$ is a maximal compact subgroup. The CHL involution $\imath_\mathrm{CHL}:  \mathrm{Nar} \to \mathrm{Nar}$ acts on the Narain lattice as the identity on $H \oplus H$ while interchanging the two summands of $E_8(-1)$. It follows that the invariant sublattice of the Narain lattice is $\mathrm{Nar}^{\langle \imath_{\text{CHL}} \rangle} \cong H \oplus H \oplus E_8(-2)$. The restriction of~(\ref{eqn:moduli}) to the corresponding $\imath_\mathrm{CHL}$-invariant symmetric space thus gives rise to a moduli space of K3 surfaces, namely the moduli space of K3 surfaces with a transcendental lattice isometric to $H \oplus H \oplus E_8(-2)$; its complement in the K3 lattice $\Lambda_\mathrm{K3} \cong H^{\oplus 3} \oplus E_8(-1)^{\oplus 2}$ is isomorphic to $H\oplus E_8(-2)$. This suggests, that the moduli space of the CHL string is the 10-dimensional moduli space of K3 surfaces polarized by the lattice $H\oplus E_8(-2)$.
\par Bershadsky et al.~analyzed the F-theory compactifications dual to the CHL string in the case of an isotrivial elliptic fibration on the F-theory background  \cite{MR1797021}:  they found that the CHL string in eight dimensions is dual to an F-theory with non-zero flux of an antisymmetric two-form field, or $B$-field, along the base curve $\mathbb{P}^1$. The value of this flux is quantized and fixed to half the K\"ahler class of $\mathbb{P}^1$. This picture was then extended to the interior of the Narain moduli space of toroidal compactifications where the elliptic fibrations of the F-theory models are no longer isotrivial. The presence of the flux freezes eight of the moduli in the physical moduli space, leaving a 10-dimensional moduli space \cite{MR1615617}. The construction of Witten is known today as \emph{the frozen phase of F-theory} \cite{MR3861184}.
\par On this moduli space, the single-valued background of an antisymmetric two-form in the physical theory is not compatible with a monodromy of the half-periods of an elliptic fiber in $\mathrm{SL}(2,\mathbb{Z})$ around the defects. Rather it must be contained in a subgroup of $\mathrm{SL}(2,\mathbb{Z})$ that keeps the flux of the $B$-field invariant. It was argued in \cite{MR1797021}  that the biggest possible being the congruence subgroup $\Gamma_0(2)$. Then, for a consistent model one has to start with an elliptically fibered K3 surface with section and a monodromy group contained in $\Gamma_0(2)$ that keeps one of the three half-periods of the elliptic fibers invariant.  The corresponding Weierstrass model indeed has eight fibers of type $I_2$ and two sections, the zero-section and a 2-torsion section, arranged in such a way that at each reducible fiber of type $A_1$ the two sections pass trough two different components of the fiber. In physics this fibration is called the \emph{$\Gamma_0(2)$ elliptic fibration} \cite{MR1797021}. 
\par We employ a different point of view that is based on recent work of the authors in \cite{Clingher:2021}. There, all Jacobian elliptic K3 surfaces with finite automorphism group were classified. Concretely,  all lattices $L$ that may occur as N\'eron-Severi lattice for a Jacobian elliptic K3 surface $\mathcal{X}$ with finite automorphism group were given, and all Jacobian elliptic fibration(s) with the reducible fibers and Mordell-Weil groups supported on a very general $L$-polarized K3 surface were determined. 
\par Given a lattice $L$, one has the \emph{discriminant group} $A_L = L^\vee/L$ and its associated discriminant form, denoted $q_L$.  A lattice $L$ is then called \emph{2-elementary} if $A_L$ is a 2-elementary abelian group, i.e., $A_L \cong (\mathbb{Z}/2\mathbb{Z})^\ell$ where $\ell=\ell_L$ is the {\it length} of L,  the minimal number of generators of the group $A_L$.  One also has the {\it parity} $\delta_L \in \{ 0,1\}$. By definition, $\delta_{L} = 0$ if $q_L(x)$ takes values in $\mathbb{Z}/2\mathbb{Z} \subset \mathbb{Q}/2\mathbb{Z}$ for all $x \in A_L$, and $ \delta_L=1$ otherwise. In this context, a result of  Nikulin \cite{MR633160}*{Thm.~4.3.2}. and \cite{MR633160b}*{Thm.~4.3.2} asserts that  even, indefinite, 2-elementary lattices (that are sublattices of the K3 lattice) are uniquely determined by their rank $\rho_L$, length $\ell_L$, and parity $\delta_L$.  
\par One consequence of the classification by the authors in \cite{Clingher:2021} is that a Jacobian elliptic K3 surface with an element of order two (or 2-torsion section) in its Mordell-Weil group and finite automorphism group requires the N\'eron-Severi lattice $L$ to be \emph{2-elementary}.  The minimal rank of $L$ then is $\rho_L =10$, and the only such lattice is precisely $L = H \oplus N$. In turn, the only Jacobian elliptic fibration that is supported on a very general $H \oplus N$-polarized K3 surface is the aforementioned fibration with the singular fibers $8 I_2 + 8 I_1$ and the Mordell-Weil group $\mathbb{Z}/2\mathbb{Z}$.  We call this fibration the \emph{alternate fibration}. In \cite{Clingher:2021}*{Table 3,4} all 2-elementary lattices $L$ were classified such that the very general $L$-polarized K3 surface has a finite automorphism group and admits an alternate fibration.  For ranks $\rho_L$ between 10 and 18, every such 2-elementary lattice $L$ supports at most one alternate fibration. This gives a precise criterion for the existence of an alternate fibration and hence a van Geemen-Sarti involution.  Note that for Picard numbers 16, 17, and 18 the corresponding lattices are
\beqn
 H \oplus E_7(-1) \oplus E_7(-1), \quad H \oplus E_8(-1) \oplus E_7(-1), \quad \text{and} \quad H \oplus E_8(-1) \oplus E_8(-1),
\eeqn 
respectively, and the corresponding alternate fibrations are well known in string theory. The advantage of this approach is that it includes \emph{all} admissible sub-varieties of the moduli spaces of $H \oplus N$-polarized K3 surfaces. For example, the lattice $H \oplus N$ admits an embedding into the lattice $L=H \oplus D_8(-1) \oplus D_4(-1)$ of rank 14. A very general $L$-polarized K3 surface admits two distinct Jacobian elliptic fibrations: the first one has the singular fibers $I_4^* + I_0^* + 8 I_1$ and a trivial Mordell-Weil group, the second fibration is the alternate fibration with the singular fibers $III^* + 5 I_2 + 5 I_1$ and the Mordell-Weil group $\mathbb{Z}/2\mathbb{Z}$. However, singular fiber of type $III^*$ have a monodromy group that is not contained in $\Gamma_0(2)$; see \cite{MR3366121}*{Table 1}. Indeed in  \cite{MR1797021} a reducible fiber of type $E_7$ was not considered problematic (as opposed to reducible fibers of type $E_6$ and $E_8$ in the alternate fibration). This demonstrates that the existence of a canonical van Geemen-Sarti involution is consistent for extensions of the lattice polarization as one restricts to natural subspaces in the moduli space of $H \oplus N$-polarized K3 surfaces whereas the  notion of $\Gamma_0(2)$-monodromy might not.
\par To our knowledge, however, the construction of  \cite{MR1797021} has not been much employed in the physics literature, and there are various criticisms that can be made of it. One such criticism is that, as suggested in the paper, one can also consider fluxes of the B-field that restrict the monodromy group to $\Gamma_0(3)$, $\Gamma_0(4)$, or $\Gamma_0(6)$. This would define other disconnected components in the moduli space of 8d string vacua with 16 supersymmetries and gauge groups of smaller rank. From the point of view of frozen singularities, however, these do not arise. Concretely, Tachikawa \cite{MR3538766} has given evidence that in F-theory there exist only frozen singularities corresponding to $I_4^*$ fibers, so that in eight dimensions one can only construct theories with rank reductions by 8 (one such fiber) and 16 (two such fibers)\footnote{The scenario has been recently claimed to be reconstructed in the Swampland Program in \cites{MR4316262,BHMV:2021}, excluding the other disconnected components mentioned above.}. Note that these cases also appear in the classification results of \cite{Clingher:2021}: the very general $H \oplus D_8(-1)$-lattice polarized K3 surface has a finite automorphism group and admits exactly one Jacobian elliptic fibration which has one singular fiber of type $I_4^*$ and a trivial Mordell-Weil group. The second case constitutes the family of $H \oplus D_8(-1) \oplus D_8(-1)$-polarized K3 surfaces which is the family of Kummer surfaces associated with two non-isogeneous elliptic curves.  In light of the work \cites{MR3538766, MR3861184} one concludes that $\mathfrak{M}_{H \oplus D_8(-1)}$ is in fact the moduli space of lattice polarized K3 surfaces corresponding to the frozen phase of F-theory.
\par In contrast, we argue that the moduli space of $H \oplus N$-polarized K3 surfaces corresponds to the dual of the CHL string in seven dimensions. The reason is that an $H \oplus N$-polarization is equivalent with the existence of a canonical van Geemen-Sarti involution as in Figure~\ref{diag:isogeny}. As explained for example in \cite{MR3538766}, if F-theory on a K3 surface $\mathcal{X}$ is further compactified on a circle $S^1$, one can orbifold the theory quotienting the $\mathcal{X} \times S^1$ by an involution $\imath_\mathcal{X}$ and a half-shift on the $S^1$; see Sections~\ref{ssec:notions} and~\ref{sec:VGS_K3} for the definitions of a Nikulin involution and van Geemen-Sarti involution.  Furthermore, this construction is dual to M-theory on K3 with two frozen singularities of type $I_0^*$, as already shown by Witten \cite{MR1615617}. The F-theory construction just mentioned is purely geometric and does not involve any flux; its M-theory dual \emph{does} involve a 3-form flux, and does not uplift to eight dimensions. In particular, this one can indeed be extended to further rank reductions; see~\cite{MR1868756} for a full treatment of this problem.
\par The identification of points in the physical moduli space corresponding to a specific non-abelian gauge group of the CHL string is then based on a second elliptic fibration with two fibers of type $I_0^*$, called the \emph{inherited elliptic fibration}~\cite{MR1797021}.  As explained in \cite{MR1615617}*{p.~31}, the two $I_0^*$ fibers appear when one compactifies further on a circle, i.e., the theory is seven dimensional. The correct F-theory compactification does not admit a section but only a bi-section due to the precise nature of the duality with the CHL string. In fact, Witten argued that the moduli space of F-theory compactifications dual to the CHL string is naturally isomorphic to the moduli space of K3 surfaces obtained as genus-one fibrations with two fibers of type $I_0^*$ and a bisection \cite{MR1615617}.  
\par In conclusion, it follows that the K3 moduli spaces $\mathfrak{M}_{H \oplus E_8(-2)}$ and $\mathfrak{M}_{H \oplus N}$ are the moduli space of the CHL string in seven dimensions and the dual F-theory, respectively. Below we shall use classical geometrical notions - algebraic correspondences, even-eight configurations, and elements of the Brauer group - to give a precise mathematical framework for the F-theory/CHL string duality.
\section{Constructions of K3 surfaces}
\label{sec:constructions}
\subsection{Notions for lattice polarized K3 surfaces}
\label{ssec:notions}
We will use the following standard notations: $L_1 \oplus L_2$ is orthogonal sum of the two lattices $L_1$ and $L_2$, $L(\lambda)$ is obtained from the lattice $L$ by multiplication of its form by $\lambda \in \mathbb{Z}$, $\langle R \rangle$ is a lattice with the matrix $R$ in some basis; $A_n$, $D_m$, and $E_k$ are the positive definite root lattices for the corresponding root systems,  $H$ is the unique even unimodular hyperbolic rank-two lattice, and $N$ is the negative definite rank-eight Nikulin lattice (for a definition, see \cite{MR728142}*{Sec.~5}).  For a lattice $L$ with a primitive embeddings $\imath: L \hookrightarrow \Lambda_\mathrm{K3}$ into the K3 lattice $\Lambda_\mathrm{K3} \cong H^{\oplus 3} \oplus E_8(-1)^{\oplus 2}$, Dolgachev proved that there exists a coarse moduli space $\mathfrak{M}_L$ of pseudo-ample $L$-polarized K3 surfaces, i.e., a moduli space of algebraic K3 surfaces $\mathcal{X}$ that are polarized by the lattice $L$ such that  $\imath(L)$ contains a numerically effective class of positive self-intersection in the N\'eron-Severi lattice $\mathrm{NS}(\mathcal{X})$. Dolgachev also established a version of mirror symmetry for the moduli spaces of lattice polarized K3 surfaces in \cite{MR1420220}: given a lattice $L$ with the property that for any two primitive embeddings $\imath_1, \imath_2 : L \hookrightarrow \Lambda_\mathrm{K3}$ there is an isometry $g \in O(\Lambda_\mathrm{K3})$ such that $\imath_2 = g \circ \imath_1$,  the isomorphism class of $\check{L}$ for a fixed splitting $L^\perp = H \oplus \check{L}$ of the orthogonal complement $L^\perp \subset\Lambda_\mathrm{K3}$  is well defined. The mirror moduli space of $\mathfrak{M}_L$ is then taken to be $\mathfrak{M}_{\check{L}}$. 
\par An important tool in relating families of K3 surfaces are Nikulin involutions. Recall that a {\it Nikulin involution} \cites{MR728142,MR544937} is an involution $\imath_{\mathcal{X}} : \mathcal{X} \rightarrow \mathcal{X}$ on a K3 surface $\mathcal{X}$ that satisfies $\imath_{\mathcal{X}}^*(\omega) = \omega $ for any holomorphic two-form $\omega$ on $\mathcal{X}$.  The minimal resolution of the quotient $\mathcal{X}/ \langle \imath_{\mathcal{X}} \rangle$ is another K3 surface $\mathcal{X}'$, and the quotient map induces a two-to-one rational map $\Phi: \mathcal{X} \dasharrow \mathcal{X}'$ whose branch locus is an even set of eight rational curves on $\mathcal{X}'$. Recall that a set of eight rational curves on a K3 surface is called an \emph{even eight} if the sum is divisible by two in the N\'eron-Severi lattice; see \cite{MR1922094}. 
\par For a \emph{Jacobian elliptic surface} $\mathcal{X}$ we denote the projection by $\pi_\mathcal{X}: \mathcal{X} \to \mathbb{P}^1$, the zero-section by $\sigma_\mathcal{X}$, and the Mordell-Weil group of sections by $\mathrm{MW}(\mathcal{X}, \pi_\mathcal{X})$ with its torsion subgroup  denoted by $\mathrm{MW}(\mathcal{X}, \pi_\mathcal{X})_\mathrm{tor}$.  A complete list of the possible singular fibers in a Weierstrass model was determined by Kodaira~\cite{MR0184257}: it encompasses two infinite families $(I_n, I_n^*, n \ge 0)$ and six exceptional cases $(II, III, IV, II^*, III^*, IV^*)$. The span of the cohomology classes associated with the elliptic fiber and the section $\sigma_\mathcal{X}$ induces a sublattice of the N\'eron-Severi lattice $\mathrm{NS}(\mathcal{X})$ isomorphic to $H$. This sublattice determines uniquely the Jacobian fibration. Conversely, a pseudo-ample lattice embedding $H \hookrightarrow \mathrm{NS}(\mathcal{X})$ determines -- up to the action of Hodge isometries of $H^2(\mathcal{X}, \mathbb{Z})$ -- a unique isomorphism class of a Jacobian elliptic fibration on $\mathcal{X}$. 
\subsection{K3 surfaces double covering a rational elliptic surface}
\label{ssec:double_covers}
For two cubics $p, q$ in $\mathbb{P}^2 =\mathbb{P}(X,Y,Z)$ the cubic pencil 
\beq
\label{eqn0:R}
 \mathcal{R} = \Big\lbrace  U p(X,Y,Z) + V q(X,Y,Z)  = 0 \Big\rbrace \ \subset \ \mathbb{P}^2 \times \mathbb{P}^1 
\eeq 
defines an elliptic surface with section $\pi_\mathcal{R}: \mathcal{R} \to \mathbb{P}^1 = \mathbb{P}(U,V)$ which is isomorphic to $\mathbb{P}^2$ blown up in the nine base points of the pencil.  The exceptional divisors are not contained in the fibers but yield independent sections.  Selecting one of them as the zero section $\sigma_\mathcal{R}$, the remaining eight generate a lattice of type $E_8$. Here, we are assuming that the two cubics $p, q$ do not have common components. Moreover, the base points of the cubic pencil may only yield independent sections if the pencil has no reducible fibers. The minimal resolution gives rise to a rational elliptic surface (RES). By the Shioda-Tate formula the rank of $\mathrm{MW}(\mathcal{R}, \pi_\mathcal{R})$ drops if there are reducible fibers.  It turns out that this characterization is complete \cite{MR986969}*{Thm. 5.6.1}, and every rational elliptic surface $\mathcal{R}$ is isomorphic to the blow-up of $\mathbb{P}^2$ in the base points of a cubic pencil.
\par The general rational elliptic surface with section $\pi_\mathcal{R}: \mathcal{R} \to \mathbb{P}^1 = \mathbb{P}(U, V)$ has the Weierstrass model
\beq
\label{eqn:R}
  \mathcal{R} : \quad Y^2 Z = X^3 +    f( U, V)  \, X Z^2 +   g(U, V)  \, Z^3 \,,
\eeq
where  $f, g$ are homogenous of degree four and six, respectively, such that the fibration has twelve singular fibers of type $I_1$ (that is, $4 f^3 + 27 g^2=0$ has twelve distinct roots), and the Mordell-Weil group is $\mathrm{MW}(\mathcal{R}, \pi_\mathcal{R}) \cong E_8$.  Here, the Mordell-Weil group can also read off in the classification of all rational elliptic surfaces given by Oguiso and Shioda in \cite{MR1104782}. The rational elliptic surface $\mathcal{R}$ with the singular fibers $12 I_1$ and $\mathrm{MW}(\mathcal{R}, \pi_\mathcal{R}) \cong E_8$ is referred to as $(1)$ in Oguiso-Shioda classification.
\par A family of K3 surfaces is obtained by a base change of order two on the rational elliptic fibration $\pi_\mathcal{R}$.  A double cover $h_{(d_0,d_\infty)}: \mathbb{P}^1=\mathbb{P}(u, v) \to \mathbb{P}^1=\mathbb{P}(U,V)$ is determined by choosing two points $[u, v]=[0:1], [1,0]$ as the branch points with images $[U, V]=[d_0 :1], [1, d_\infty]$ with $d_0 d_\infty \not = 1$, and an additional point, say $[u, v]=[1:1]$ with image $[U, V]=[1 :1]$. The double cover map is then given by
\beq
\label{eqn:double_cover}
 h_{(d_0,d_\infty)}: \ \Big[ u : v \Big] \ \mapsto \   \Big[ U: V \Big] \ = \  \Big[ \big(1 - d_0\big) u^2 + d_0 \big(1 - d_\infty) v^2 : \, d_\infty \big(1 - d_0\big) u^2 + \big( 1 - d_\infty\big) v^2 \Big] \,.
\eeq
Here, we are assuming that the fibers of~(\ref{eqn:R}) over the points $[d_0 :1], [1, d_\infty], [1:1]$ are not singular. Geometrically, $h_{(d_0,d_\infty)}$ determines a line bundle $\mathcal{L} \to \mathbb{P}^1$ such that $\mathcal{L}^{\otimes 2} = \mathcal{O}_{\mathbb{P}^1}$. Pulling back the Weierstrass model in Equation~(\ref{eqn:R}), we obtain a Jacobian elliptic K3 surfaces $\pi_\mathcal{Y}: \mathcal{Y} \to \mathbb{P}^1 = \mathbb{P}(u, v)$ given by
 \beq
 \label{eqn:Y}
  \mathcal{Y}: \quad y^2 z = x^3 +    F( u, v) \,   x z^2 +   G( u, v)   \, z^3 \,,
\eeq
where we have set $F = h^*_{(d_0,d_\infty)}f$ and $G= h^*_{(d_0,d_\infty)}g$. We have the following:
\begin{lemma}
\label{lem:Y}
The very general K3 surface $\mathcal{Y}$ in Equation~(\ref{eqn:Y}) has 24 singular fibers of type $I_1$ and a Mordell-Weil group $\mathrm{MW}(\mathcal{Y}, \pi_\mathcal{Y}) \cong E_8(2)$.
\end{lemma}
\begin{proof}
In the formula for the height pairing on an elliptically fibered surface with section all summands -- the holomorphic Euler characteristic, the intersection number, the number of singular fibers -- double when pulling back the Weierstrass equation via Equation~(\ref{eqn:double_cover}).
 \end{proof}
In particular, the very general K3 surface $\mathcal{Y}$ has $\mathrm{NS}(\mathcal{Y}) \cong H \oplus E_8(-2)$. Conversely, it is known that every element in $\mathfrak{M}_{H \oplus E_8(-2)}$  is represented by a family of double covers of the plane branched over the union of cubics, i.e., double covers of the rational elliptic surface obtained by the minimal resolution of a reducible double-sextic blown-up in nine points \cite{MR1420220}*{Sec.~9.1}.  To find the moduli from the Weierstrass model in Equation~(\ref{eqn:Y}), we observe that $f$ and $g$ depend on $5 + 7 =12$ parameters, and there are two additional parameters $d_0, d_\infty$. Using a transformation of the type
\beq
\label{eqn:rescalingHirzebruch}
 \Big( \big(u :v \big) , \; \big(x:y:z\big) \Big) \ \mapsto \   \Big( \big(\lambda u :\, \lambda v\big) , \; \big( \lambda^2 x: \lambda^3 y:z\big ) \Big) 
 \eeq
 with $\lambda \in \mathbb{C}^\times$ and the automorphism group $\mathrm{PGL}(2,\mathbb{C})$ of  $\mathbb{P}^1$ we obtain $12 +2 -1 -3 =10$ as the number of moduli.  In fact, the following was proved in \cite{MR4069236}*{Prop.~4.17}:
\begin{proposition}
\label{prop1}
The K3 surfaces in Equation~(\ref{eqn:Y}) form the 10-dimensional moduli space $\mathfrak{M}_{H \oplus E_8(-2)}$.
\end{proposition}  
\par Conversely, the rational elliptic surface $\mathcal{R}$ is the minimal resolution of the quotient $\mathcal{Y}/ \langle k_\mathcal{Y} \rangle$ where $k_\mathcal{Y}$ is the antisymplectic involution preserving the Jacobian elliptic fibration induced by the deck transformation $[u:v] \mapsto [-u:v]$. A Nikulin involution $\jmath_\mathcal{Y}=k_\mathcal{Y} \circ (-1)$ is given by the composition of $k_\mathcal{Y}$ with the (antisymplectic) hyperelliptic involution acting as $(-1)$ on the fibers.  The minimal resolution of $\mathcal{Y}/ \langle \jmath_\mathcal{Y}\rangle$ then yields another Jacobian elliptic K3 surface $\pi_{\widetilde{\mathcal{Y}}}: \widetilde{\mathcal{Y}} \to \mathbb{P}^1 = \mathbb{P}(U, V)$ with the Weierstrass model
 \beq
  \label{eqn:Yprime}
  \begin{split}
  \widetilde{\mathcal{Y}}: \quad Y^2 Z & = X^3   +   (U - d_0 V)^2  (d_\infty U -  V)^2 f( U, V)   \, X Z^2\\
   &+   (U - d_0 V)^3  (d_\infty U -  V)^3  g( U, V) \, Z^3 \,.
  \end{split} 
\eeq
We have the following:
\begin{lemma}
\label{lem:Yprime}
The very general K3 surface $\widetilde{\mathcal{Y}}$ in Equation~(\ref{eqn:Yprime}) has 2 singular fibers of type $I_0^*$, 12 singular fibers of type $I_1$, and a trivial Mordell-Weil group.
\end{lemma}
Conversely, starting with a K3 surface $\widetilde{\mathcal{Y}}$ with the given singular fibers and Mordell-Weil group, a double cover is obtained as the pull-pack via the degree-2 rational map branched over an obvious even eight, namely the even eight given by the non-central components of the two reducible fibers of type $D_4$.  We have the following:
 \begin{proposition}
 \label{prop4}
The K3 surfaces in Equation~(\ref{eqn:Yprime}) form the 10-dimensional moduli space $\mathfrak{M}_{H \oplus D_4(-1)^{\oplus 2}}$.
\end{proposition}  
\par The results of Proposition~\ref{prop1} and Proposition~\ref{prop4} also hold (for suitably modified lattices) if we restrict $\mathcal{R}$ to a rational elliptic surface with reducible fibers given in the classification by Oguiso and Shioda in \cite{MR1104782}. Geometrically, these rational elliptic surfaces arise if the base points of the pencil in Equation~(\ref{eqn0:R}) are not distinct.  We will focus on the cases when the Mordell-Weil group of the rational elliptic surface has half the rank and is isomorphic to $D^\vee_4$ (up to torsion). That is, we will consider cases $(9)$ or $(13)$ in the Oguiso-Shioda classification \cite{MR1104782}, i.e., the rational elliptic surfaces with the singular fibers and Mordell-Weil groups given by
\beq
\label{eqn:casesRES}
 I_0^* + 6 I_1, \, \mathrm{MW}(\mathcal{R}, \pi_\mathcal{R}) \cong D_4^\vee, \qquad 4 I_2 + 4 I_1,\,  \mathrm{MW}(\mathcal{R}, \pi_\mathcal{R}) \cong \mathbb{Z}/2\mathbb{Z} \oplus D_4^\vee,
\eeq 
respectively. We have the following:
\begin{proposition}
\label{cor:2coverK3_14}
 \leavevmode
\begin {enumerate}
\item Consider the rational elliptic surfaces $\mathcal{R}$ associated with labels $(9)$ and $(13)$ in the Oguiso-Shioda classification. Let $\mathcal{Y}$ be the K3 surfaces of  Equation~(3.7) obtained via a base-change of order two from $\mathcal{R}$. Then, in the very general case, the Neron-Severy lattice $\mathrm{NS}(\mathcal{Y})$ is  $H \oplus K_0(-1)$ or $H \oplus N_0(-1)$, respectively. Here, $K_0$ is the positive definite lattice of rank $12$, determinant $2^6$, and the Gram matrix
\beq
\label{eqn:Gram_matrix_1}
\scalemath{0.8}{
 \left( \begin{array}{rrrrrrrrrrrrrrr} 
4 	&-1	& 1	&  1	& -1	&  1	&  1	& -1	&  1	&  1 & -1	& 1 \\
 -1  	&2 	&-1  	& 0  	& 0  	& 0   & 0  	& 0  	& 0  	& 0 	& 0  	& 0 \\
  1 	& -1 	& 2  	& 0  	& 0  	& 0   & 0  	& 0  	& 0  	& 0 	& 0  	& 0 \\
 1	& 0	& 0	& 2	& -1 	& 0   & 0  	& 0  	& 0  	& 0 	& 0  	& 0 \\
 -1	& 0	& 0	& -1	& 2 	& -1  & 0  	& 0  	& 0  	& 0 	& 0  	& 0 \\
  1	& 0	& 0	& 0	& -1 	& 2  & 0  	& 0  	& 0  	& 0 	& 0  	& 0 \\
  1	& 0	& 0	& 0	& 0 	& 0  & 2  	& -1  & 0  	& 0 	& 0  	& 0 \\
  -1	& 0	& 0	& 0	& 0 	& 0  & -1  	& 2  	& -1  & 0 	& 0  	& 0 \\
  1	& 0	& 0	& 0	& 0 	& 0  & 0  	& -1  & 2 	& 0 	& 0  	& 0 \\
  1	& 0	& 0	& 0	& 0 	& 0  & 0  	& 0 	& 0 	& 2	& -1 	& 0 \\
  -1	& 0	& 0	& 0	& 0 	& 0  & 0  	& 0 	& 0 	& -1	& 2 	& -1 \\
  1	& 0	& 0	& 0	& 0 	& 0  & 0  	& 0 	& 0 	& 0	& -1 	& 2 
\end{array}\right),} 
\eeq
and $N_0$ is the positive definite lattice of rank $12$, determinant $2^8$, and the Gram matrix
\beq
\scalemath{0.8}{
\label{eqn:Gram_matrix_2}
 \left( \begin{array}{rrrrrrrrrrrr} 
 4 & 1 &1 &1 &1 &1 &1 &1  &-2  &-2  &-2  &-2\\
 1 & 2  &0  &0  &0  &0  &0  &0  &1  &1  &1  &0\\
 1  &0 &2  &0  &0  &0  &0  &0  &-1  &-1  &-1  &-1\\
 1  &0  &0 &2  &0  &0  &0  &0  &-1  &-1  &-1  &-1\\
 1  &0  &0  &0 &2  &0  &0  &0  &0  &0  &0  &-1\\
 1  &0  &0  &0  &0 &2  &0  &0  &0  &0  &0  &-1\\
 1  &0  &0  &0  &0  &0 &2  &0  &0  &0  &0  &0\\
 1  &0  &0  &0  &0  &0  &0 &2  &0  &0  &0  &0\\
-2  &-1  &-1  &-1  &0  &0  &0  &0 &4 &2 &2  &0\\
-2  &-1  &-1  &-1  &0  &0  &0  &0 &2 &4 &2  &0\\
-2  &-1  &-1  &-1  &0  &0  &0  &0 &2 &2 &4  &0\\
-2  &0  &-1  &-1  &-1  &-1  & \ 0   & \ 0 &0  &0  &0 &4
\end{array}\right).}
\eeq
In particular the associated moduli spaces for $\mathcal{Y}$ are 6-dimensional, given by $\mathfrak{M}_L$  with $L = H \oplus K_0(-1)$ and $L= H \oplus N_0(-1)$, respectively. 
\item The 6-dimensional moduli spaces $\mathfrak{M}_L$ for $L = H \oplus D_4(-1)^{\oplus 3}$ and $L= H \oplus D_8(-1) \oplus  A_1(-1)^{\oplus 4}$  are given by K3 surfaces $\widetilde{\mathcal{Y}}$  in Equation~(\ref{eqn:Yprime})  obtained  as the minimal resolution of $\mathcal{Y}/ \langle j_\mathcal{Y} \rangle$ where $j_\mathcal{Y}$ is the Nikulin involution constructed above and $\mathcal{Y}$ are chosen from the first and second family in (1), respectively.
\end{enumerate}
These cases are summarized in Table~\ref{tab:K3JEF4}.
\end{proposition}
\begin{proof}
The singular fibers and Mordell-Weil groups of the K3 surfaces $\mathcal{Y}$ are $2 I_0^* + 12 I_1$ and $D_4^\vee(2)$  (which is isometric to $D_4$) and $8 I_2 + 8I_1$ and $\mathbb{Z}/2\mathbb{Z} \oplus D_4^\vee(2)$, respectively. Thus,  $\mathrm{NS}(\mathcal{Y})$ have determinant $2^6$ and $2^8$, respectively. In the first case, $\mathrm{NS}(\mathcal{Y})$ is isomorphic to the polarizing lattice of the surface $\mathcal{X}'$ given in Table~\ref{tab:WEQ_2}, which, in turn, also admits an elliptic fibration with section, singular fibers of type $4 I_4 + 8 I_1$, and Mordell-Weil group $\mathbb{Z}/2\mathbb{Z}$. This follows from the fact that $\mathcal{G}$ in Table~\ref{tab:WEQ_2} admits sections for the elliptic fibrations induced by the projection both onto $\mathbb{P}(u, v)$ and $\mathbb{P}(s, t)$. It follows that $\mathrm{NS}(\mathcal{Y})$ is of the form $H \oplus K_0(-1)$ and the overlattice of $H \oplus A_3(-1)^{\oplus 4}$ associated with $\vec{v} = \langle 0, 0 | 1, 0, 1| \dots | \dots | 1, 0, 1\rangle/2$. The latter is equivalent to the class of the 2-torsion section.  A Gram matrix computation yields (\ref{eqn:Gram_matrix_1}). 
\par In the second case, one works in the context of Figure~\ref{fig:FTH-CHLproof}, with explicit equations for all surfaces involved given in Table~\ref{tab:WEQ}. The K3 surface $\mathcal{G}$ is described as a double cover of $\mathbb{P}(s, t) \times \mathbb{P}(u, v)$, with both projections inducing elliptic fibrations with section. The relative Jacobians of the two elliptic fibrations are $ \mathcal{Y} $ and $ \mathcal{X}' $, respectively. As the two elliptic fibrations have sections and the relative Jacobian map is birational, one obtains that K3 surfaces  $\mathcal{G}$, $ \mathcal{Y} $ and $ \mathcal{X}'$ are in fact mutually isomorphic. In particular, one has $\mathrm{NS}(\mathcal{X}') \cong \mathrm{NS}(\mathcal{G}) \cong \mathrm{NS}(\mathcal{Y})$. By construction, $\mathcal{X}$ and $\mathcal{F}$ are isomorphic surfaces and, in particular, one has $ \mathrm{NS}(\mathcal{F}) \cong \mathrm{NS}(\mathcal{X})$. In addition, $\mathcal{X}$ and $ \mathcal{X}' $ both carry Jacobian elliptic fibrations with singular fiber type $8 I_1 + 8 I_2$ and Mordell-Weil group $\mathbb{Z}/2\mathbb{Z} \oplus D_4$.  This implies $\mathrm{NS}(\mathcal{X}') \cong \mathrm{NS}(\mathcal{X}) $. One obtains therefore that surfaces $\mathcal{X}$, $ \mathcal{X}' $, $\mathcal{G}$, $ \mathcal{Y} $, $\mathcal{F}$ carry isometric Neron-Severi lattices. 
\par Using the Jacobian elliptic fibration on $\mathcal{X}$, with singular fiber type $8 I_1 + 8 I_2$ and Mordell-Weil group $\mathbb{Z}/2\mathbb{Z} \oplus D_4$, one may construct the N\'eron-Severi lattice explicitly. Note that the zero section $S_0$ and 2-torsion section $S_1$, which span $\operatorname{MW}(\mathcal{X}, \pi_\mathcal{X})_{\mathrm{tor}}$, intersect the neutral and the non-neutral components of each reducible fiber of type $A_1$, respectively. Moreover, via the formulas for $\mathcal{X}$ and $\mathcal{F}$ inTable~\ref{tab:WEQ}, one obtains explicitly four disjoint non-torsion sections $R_i$ with $i=1, \dots, 4$, by cutting out:
\beqn
 X = \gamma(s^2, t^2) U_i, \quad Z=V_i, \quad Y^2 = \frac{U_i}{4 c(U_i, V_i)} \gamma(s^2, t^2)^2 \big(2 c(U_i, V_i) s^2 + a(U_i, V_i) t^2\big)^2 ,
\eeqn
where $[U_i: V_i]$ are solutions of $a(U_i, V_i)^2 -4 c(U_i, V_i) d(U_i, V_i)=0$. We note that each section $R_i$ intersects the neutral components of the fibers over $\delta(s^2, t^2)=0$ and the non-neutral components over $\gamma(s^2, t^2)=0$. One can similarly construct a new set of sections $\tilde{R}_i$ for $i=1, \dots, 4$, by choosing an alternate factorization $\gamma(s^2, t^2)\delta(s^2, t^2) = \tilde{\gamma}(s^2, t^2) \tilde{\delta}(s^2, t^2)$. In this context, note that if $\tilde{\gamma}$ (and  $\tilde{\delta}$) is chosen such that  $\tilde{\gamma}(s^2, t^2)$ has exactly one pair of roots in common with $\gamma(s^2, t^2)$ and one pair of roots in common with $\delta(s^2, t^2)$, then the four sections $(R_1, R_2, R_3, \tilde{R}_4)$ generate the non-torsion part of the Mordell-Weil group and their height-pairing matrix is the Cartan matrix for $D_4$. 
\par The classes
\beq
 \langle F, \, F+S_0, \ a_1, \dots, a_8, \ R_1 -F -2 S_0, \, R_2-F -2 S_0, \,  R_3 -F -2 S_0, \,  \tilde{R}_4 -F -2 S_0 \rangle
\eeq 
span then a lattice $H\oplus \widetilde{N}_0$, where $a_j$ are generators for the non-neutral components of the fibers of type $A_1$ and $F$ is the fiber class. It follows that $\mathrm{NS}(\mathcal{X})$ is the overlattice of $H\oplus \widetilde{N}_0$ associated with $\vec{w} = \langle 0, 0| 1, \dots, 1 | 0, 0, 0, 0\rangle/2$. The latter is equivalent to the class of $S_1$. One obtains that $\mathrm{NS}(\mathcal{X})$ is of the form $H \oplus N_0(-1)$, and a Gram matrix computation yields (\ref{eqn:Gram_matrix_2}). 
\begin{figure}
\beqn
\xymatrix{
 \mathcal{X}'  \ar@<-2pt>[d]  &  \mathcal{G} \ar[l]_{\cong} \ar[dl]  \ar[d]  \ar[r]^{\cong} & \mathcal{Y}  \ar[d] \ar[dr]\\
 \mathcal{X}  \ar@<-2pt>[u]  &  \mathcal{F} \ar[l]_{\cong}   \ar[r]^{\mathrm{Jac}^0} & \widetilde{\mathcal{Y}} &  \mathcal{R}\\
 }
\eeqn
\caption{The F-theory/CHL string duality in rank 14}
\label{fig:FTH-CHLproof}
\end{figure}
\end{proof}
\par The importance of the subfamilies above stems from their geometric interpretation: the two families in Proposition~\ref{cor:2coverK3_14}~(1) can be understood as imposing the existence of another compatible involution.  We consider the two cases when this additional involution, compatible with the elliptic fibration, is antisymplectic or symplectic.
\par Let us assume that the Weierstrass model in Equation~(\ref{eqn:Y}) admits a second commuting \emph{antisymplectic involution} preserving the Jacobian elliptic fibration. We use a transformation in $\mathrm{PGL}(2,\mathbb{C})$ to have $d_0 = d_\infty=0$. We can also assume that the second antisymplectic involution is induced by the deck transformation $[u:v] \mapsto [v:u]$. Thus, without loss of generality, a family of K3 surfaces $\mathcal{Y}'$ admitting two commuting antisymplectic involutions preserving the Jacobian elliptic fibration is given by the Weierstrass model
 \beq
 \label{eqn:Z}
  \mathcal{Y}': \quad Y^2 Z = X^3 +    f\Big( (\tilde{u}^2-\tilde{v}^2)^2, \,  (\tilde{u}^2+\tilde{v}^2)^2\Big) \,   X Z^2 +   g\Big( (\tilde{u}^2-\tilde{v}^2)^2, \,  (\tilde{u}^2+\tilde{v}^2)^2\Big)    \, Z^3 ,
\eeq
where $f$ and $g$ are homogenous of degree 2 and 3, respectively, and we changed the variables from $(u, v)$ to $(\tilde{u}, \tilde{v})$ to avoid any conflict of notation. To find the moduli from the Weierstrass model in Equation~(\ref{eqn:Z}), we observe that $f$ and $g$ depend on $3 + 4 =7$ parameters. With only the freedom of a transformation as in Equation~(\ref{eqn:rescalingHirzebruch}) remaining, we obtain $7 -1 =6$ as the number of moduli in Equation~(\ref{eqn:Z}). Thus, we have the following:

\begin{proposition}
\label{prop7}
The K3 surfaces in Equation~(\ref{eqn:Z}) form a 6-dimensional subvariety of $\mathfrak{M}_{H \oplus E_8(-2)}$ whose general member admits an additional antisymplectic involution induced by an involution on the base curve.
\end{proposition}  
\par The minimal resolution of $\mathcal{Y}'/ \langle \jmath_1 \rangle$ (where $k_1$ is induced by $[\tilde{u}:\tilde{v}] \mapsto [\tilde{u}: -\tilde{v}]$ and $\jmath_1=k_1 \circ (-1)$) yields the Jacobian elliptic K3 surface $\pi_\mathcal{Y}: \mathcal{Y} \to \mathbb{P}^1 = \mathbb{P}(u, v)$ with the Weierstrass model 
 \beq
  \label{eqn:Y_sub}
  \mathcal{Y}: \quad y^2 z  = x^3   +   (u^2- v^2)^2   f( u^2, v^2)   \, x z^2 +    (u^2- v^2)^3  g( u^2, v^2)  \, z^3 .
\eeq
The image of the (second)  involution $k_2$ on  $\mathcal{Y}'$ (where $k_2$ is induced by $[\tilde{u}: \tilde{v}] \mapsto [\tilde{v} : \tilde{u}]$) is an antisymplectic involution $k_\mathcal{Y}$ preserving the Jacobian elliptic fibration on $\mathcal{Y}$ and is induced by $[u:v] \mapsto [-u:v]$.  For the Nikulin involution $\jmath_\mathcal{Y}=k_\mathcal{Y} \circ (-1)$ (which is the image of $\jmath_2 =k_2 \circ (-1)$), the minimal resolution of $\mathcal{Y}/ \langle \jmath_\mathcal{Y}\rangle$ yields another Jacobian elliptic K3 surface $\pi_{\widetilde{\mathcal{Y}}}: \widetilde{\mathcal{Y}} \to \mathbb{P}^1 = \mathbb{P}(U, V)$ with the Weierstrass model
 \beq
  \label{eqn:Yprime_sub}
  \widetilde{\mathcal{Y}}: \quad Y^2 Z  = X^3   + U^2 V^2 (U-V)^2 f(U, V)  \, X Z^2
   +     U^3 V^3 (U-V)^3 g(U, V)\, Z^3 .
\eeq
Similarly, one checks that the minimal resolution of $\mathcal{Y}/ \langle k_\mathcal{Y}\rangle$ yields the rational elliptic surface with the Weierstrass model 
 \beq
  \label{eqnYtildetildeR}
   Y^2 Z  = X^3   +  (U-V)^2 f( U, V)   \, X Z^2
   +     (U-V)^3 g(U, V) \, Z^3 \,,
\eeq
which realizes the first case in~(\ref{eqn:casesRES}). Thus, we have the following:
\begin{proposition}
\label{prop5}
The K3 surfaces in Equation~(\ref{eqn:Y_sub}) and~(\ref{eqn:Yprime_sub}) are the minimal resolutions of $\mathcal{Y}' / \langle \jmath_1 \rangle$ and $\mathcal{Y}' / \langle \jmath_1 , \jmath_2 \rangle$ and form the 6-dimensional moduli spaces $\mathfrak{M}_L$ for $L = H \oplus K_0(-1)$  (with the Gram matrix of $K_0$ given in Equation~(\ref{eqn:Gram_matrix_1})) and $L= H \oplus D_4(-1)^{\oplus 3}$, respectively. 
\end{proposition}
\par As for the second case in~(\ref{eqn:casesRES}), we remark that the existence of a 2-torsion section on $\mathcal{R}$ guarantees the existence of a 2-torsion section on $\mathcal{Y}$ in Equation~(\ref{eqn:Y}). In turn, the 2-torsion section yields an additional \emph{symplectic involution} on $\mathcal{Y}$ known as van~Geemen-Sarti involution -- we will explain the construction of a van~Geemen-Sarti involution in detail in Section~\ref{sec:VGS_K3}.
\begin{table}
\scalemath{0.9}{
\begin{tabular}{c|ll:ccc}
$\rho$ & \multicolumn{2}{|c:}{surface} & polar. lattice $L$ & sing. fibers &  Mordell-Weil grp. \\
\hline
10	& K3 	& $\mathcal{Y}$			& $H \oplus E_8(-2)$ 			& $24 I_1$ 		& $E_8(2)$ \\
10	& K3 	& $\widetilde{\mathcal{Y}}$ 	& $H \oplus D_4(-1)^{\oplus 2}$	& $2I_0^* + 12 I_1$	& $\{ \mathbb{I} \}$\\
& RES: (1)	& $\mathcal{R}$ 			& --							& $12 I_1$		& $E_8$\\
\hline	
14	& K3 	& $\mathcal{Y}$ 			& $H \oplus K_0(-1)$				& $2I_0^* + 12 I_1$	&  $D_4$ \\
14	& K3 	& $\widetilde{\mathcal{Y}}$ 	& $H \oplus D_4(-1)^{\oplus 3}$	& $3I_0^* + 6 I_1$	& $\{ \mathbb{I} \}$\\
& RES: (9)	& $\mathcal{R}$ 			& --							& $I_0^* + 6 I_1$	& $D_4^\vee$\\
\hdashline
14	& K3 	& $\mathcal{Y}$ 			& $H \oplus N_0(-1)$			& $8 I_2 + 8 I_1$ 		&  $\mathbb{Z}/2\mathbb{Z} \oplus D_4$ \\
14	& K3 	& $\widetilde{\mathcal{Y}}$ 	& $H \oplus D_8(-1) \oplus A_1(-1)^{\oplus 4}$	& $2I_0^* + 4 I_2 + 4 I_1$	& $\mathbb{Z}/2\mathbb{Z}$\\
& RES: (13)	& $\mathcal{R}$ 			& --							& $4 I_2 + 4 I_1$	& $\mathbb{Z}/2\mathbb{Z} \oplus D_4^\vee$\\
\hline
\end{tabular}}
\captionsetup{justification=centering}
\caption{K3 lattices and Jacobian elliptic fibrations in Prop.~\ref{prop1}-\ref{prop4}}
\label{tab:K3JEF4}
\end{table}
\subsection{K3 surfaces with van Geemen-Sarti involution}
\label{sec:VGS_K3}
When a K3 surface $\mathcal{X}$ admits a Jacobian elliptic fibration with a 2-torsion section, then it admits a special Nikulin involution, called a \emph{van~Geemen-Sarti involution}; see  \cite{MR2274533}. When quotienting by this involution, denoted by $\jmath_\mathcal{X}$, and blowing up the fixed locus, one obtains a new K3 surface $\mathcal{X}'$ together with a rational double cover map $ \Phi_\mathcal{X} : \mathcal{X} \dashrightarrow \mathcal{X}'$. In general, a van~Geemen-Sarti involution $\jmath_\mathcal{X}$ does not determine a Hodge isometry between the transcendental lattices $\mathrm{T}_\mathcal{X}(2)$ and $\mathrm{T}_{\mathcal{X}'}$. Van Geemen-Sarti involutions always appear as fiber-wise translation by 2-torsion in a suitable Jacobian elliptic fibration $\pi_\mathcal{X} : \mathcal{X} \to \mathbb{P}^1$ which we call the \emph{alternate fibration}; see \cite{MR3995925} for the nomenclature. In the Mordell-Weil group $\mathrm{MW}(\mathcal{X}, \pi_\mathcal{X})$ with identity element $\sigma_\mathcal{X}$, let $\tau_\mathcal{X}$  denote the  2-torsion section such that translation by $\tau_\mathcal{X}$ is the involution $\jmath_\mathcal{X}$. Moreover, the construction induces a Jacobian elliptic fibration $\pi_{\mathcal{X}'}: \mathcal{X}' \to \mathbb{P}^1$ on $\mathcal{X}'$ which also admits a 2-torsion section. Thus, we obtain Figure~\ref{diag:isogeny}. We will refer to the construction of Figure~\ref{diag:isogeny} as \emph{van~Geemen-Sarti-Nikulin duality}. If the Mordell-Weil group contains $(\mathbb{Z}/2\mathbb{Z})^2$, then the fiberwise translations by the different 2-torsion sections constitute commuting van Geemen-Sarti involutions compatible with the given Jacobian elliptic fibration, or,  commuting van Geemen-Sarti involutions for short.
\begin{figure}
\beqn
\xymatrix{
\mathcal{X} \ar @(dl,ul) ^{\jmath_\mathcal{X} } \ar [dr] _{\pi_\mathcal{X} } \ar @/_0.5pc/ @{-->} _{\Phi_\mathcal{X}} [rr]
&
& \mathcal{X}' \ar @(dr,ur) _{\jmath_{\mathcal{X}'} } \ar [dl] ^{\pi_{\mathcal{X}'} } \ar @/_0.5pc/ @{-->} _{\Phi_{\mathcal{X}'}} [ll] \\
& \mathbb{P}^1 }
\eeqn
\caption{K3 surfaces connected by van~Geemen-Sarti-Nikulin duality}
\label{diag:isogeny}
\end{figure}
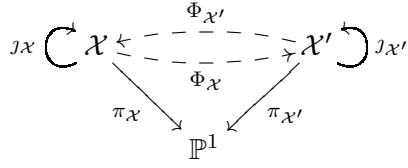
\par The K3 surface $\mathcal{X}$ has the Weierstrass equation
\beq
\label{eqn:X}
 \mathcal{X}: \quad Y^2 Z = X \Big( X^2 - A(s, t) \, XZ + B(s, t) \, Z^2 \Big) \,,
\eeq 
where $[s:t] \in \mathbb{P}^1$, $[X:Y:Z]\in \mathbb{P}^2$, $A(s, t)$ and $B(s, t)$ are homogeneous polynomials of degree four and eight, respectively, and the sections $\sigma_\mathcal{X}, \tau_\mathcal{X}$ are given by the section at infinity and $[X:Y:Z]=[0:0:1]$. To find the moduli from the Weierstrass model in Equation~(\ref{eqn:X}), we note that $A$ and $B$ depend on $5 + 9 =14$ parameters. Using transformations of the type $(X,Y) \mapsto (\lambda^2 X, \lambda^3 Y)$ with $\lambda \in \mathbb{C}^\times$ and the automorphism group $\mathrm{PGL}(2,\mathbb{C})$ of  $\mathbb{P}^1$ we get $14 -1 -3 =10$ moduli. Since we have identified coordinates according to
\beqn
 \Big((s, t) , \, (X,Y,Z)\Big) \ \sim \  \Big((\lambda s, \lambda t) , \, (\lambda^4 X, \lambda^6 Y,Z)\Big)  \,,
\eeqn
Equation~(\ref{eqn:X}) defines a double cover of the Hirzebruch surface $\mathbb{F}_4$. Similarly, the K3 surface $\mathcal{X}'$ has the Weierstrass model
\beq
\label{eqn:Xprime}
 \mathcal{X}': \quad y^2 z = x \Big( x^2 + 2 A(s, t) \, x z+ \big(A(s, t)^2 - 4 B(s, t)\big) \, z^2 \Big) \,.
\eeq
Explicit equations for the rational maps $\Phi_\mathcal{X}$ and $\Phi_{\mathcal{X}'} $ in Figure~\ref{diag:isogeny} were given in \cite{MR3995925}. The discriminant functions of the two elliptic fibrations are as follows:
\beq
\label{eqn:DiscrX_Xp}
 \Delta_\mathcal{X} = B(s, t)^2 \big(A(s, t)^2 - 4 B(s, t)\big)\,, \qquad  \Delta_{\mathcal{X}'} = 16 B(s, t) \big(A(s, t)^2 - 4 B(s, t)\big)^2\,.
\eeq
\begin{lemma}
\label{lem:X}
The very general K3 surface $\mathcal{X}$ in Equation~(\ref{eqn:X}) and $\mathcal{X}'$ in Equation~(\ref{eqn:Xprime}) have 8 fibers of type $I_1$ over the zeroes of $A^2 - 4 B=0$ (resp.~$B=0$) and 8 fibers of type $I_2$ over the zeroes of $B=0$ (resp.~$A^2 - 4 B=0$) with $\mathrm{MW}(\mathcal{X}, \pi_\mathcal{X}) \cong \mathrm{MW}(\mathcal{X}', \pi_{\mathcal{X}'}) \cong \mathbb{Z}/2\mathbb{Z}$.
\end{lemma}
In particular, the very general Jacobian elliptic K3 surfaces $\mathcal{X}$ in Equation~(\ref{eqn:X}) and $\mathcal{X}'$  in Equation~(\ref{eqn:Xprime}) have
\beq
\label{eqn:lattice}
  \mathrm{NS}(\mathcal{X}) \cong \mathrm{NS}(\mathcal{X}') \cong H \oplus N \,, \qquad
  \mathrm{T}_\mathcal{X} \cong \mathrm{T}_{\mathcal{X}'}  \cong H^2 \oplus N \,.
\eeq
Thus, we have the following:
 \begin{proposition}
 \label{prop6}
The K3 surfaces in Equation~(\ref{eqn:X}) form the 10-dimensional moduli space $\mathfrak{M}_{H \oplus N}$.
\end{proposition}  
Geometrically, the K3 surface $\mathcal{X}$ is a double cover of $\mathcal{X}'$ (via the rational map $\Phi_\mathcal{X}$) branched over the even eight that consists of the eight components of the fibers of type $I_2$ that are not met by the zero section $\sigma_{\mathcal{X}'} $, i.e., the eight exceptional curves in the corresponding reducible fibers of type $A_1$. The other components, which meet $\sigma_{\mathcal{X}'} $ map 2:1 to components in $\mathcal{X}$ where they are interchanged and also the two singular points of the fiber are permuted resulting in $I_1$-type fibers on $\mathcal{X}$. The fixed points of the translation by $\tau_\mathcal{X}$ are the eight nodes in the $I_1$-type fibers, blowing them up gives $I_2$-type fibers in $\mathcal{X}'$.  Similarly, the K3 surface $\mathcal{X}'$ is a double cover of $\mathcal{X}$ (via the rational map $\Phi_{\mathcal{X}'}$).  The eight exceptional curves $\mathit{C}_1, \mathit{C}_2, \dots, \mathit{C}_8$ resulting on $\mathcal{X}$ from the singularity resolution form an \emph{even-eight configuration}  \cite{MR1922094}, i.e.
\beq
\label{eq:even_eight}
  \frac{1}{2} \ \left ( \mathit{C}_1 + \mathit{C}_2 + \cdots +  \mathit{C}_8 \right )  \ \in \  \operatorname{NS}(\mathcal{X}) \,.
 \eeq 
This configuration of eight curves whose formal sum is in $2\operatorname{NS}(\mathcal{X})$ is known to determine in full the double cover $\Phi_{\mathcal{X}'}: \mathcal{X}' \dashrightarrow \mathcal{X}$; see  \cite{MR1922094}. Concretely, each reducible fiber in the Jacobian elliptic fibration $\pi: \mathcal{X} \to \mathbb{P}(s, t)$ consists of two components $\mathit{F}_{i0}$ and $\mathit{F}_{i1}$ such that
\beq
 \mathit{F}_{i0} \circ \sigma_\mathcal{X} =1, \quad \mathit{F}_{i0} \circ \tau_\mathcal{X} =0, \qquad 
 \mathit{F}_{i1} \circ \sigma_\mathcal{X} =0, \quad \mathit{F}_{i1} \circ \tau_\mathcal{X} =1,
\eeq 
for $i=1, \dots, 8$. The van Geemen-Sarti involution interchanges $\mathit{F}_{i0}$ and $\mathit{F}_{i1}$ for $i=1, \dots, 8$. The eight curves $\mathit{F}_{i1}$ for $i=1, \dots, 8$ form an even eight that is not met by the zero section $\sigma_\mathcal{X}$; see Figure~\ref{fig:EFS}. Thus, the double cover $\mathcal{X}'$ obtained from the double cover branched on $\mathit{F}_{11} + \cdots + \mathit{F}_{81}$ is elliptically fibered with section $\sigma_{\mathcal{X}'}$ and the two-torsion section $\tau_{\mathcal{X}'}$; the two sections form the preimage of $\sigma_\mathcal{X}$ under $\Phi_{\mathcal{X}'}$.
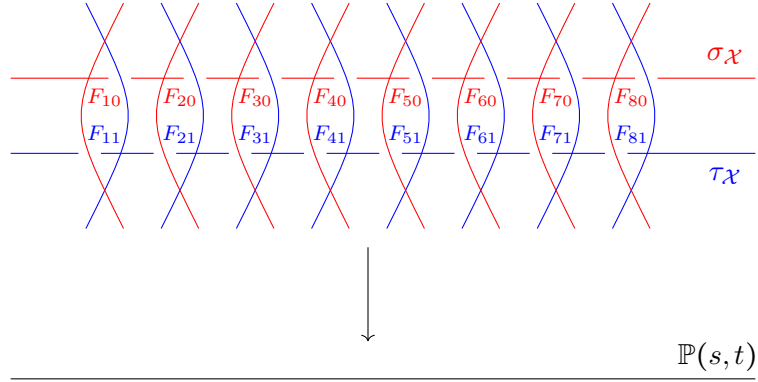
\begin{figure}
\centering
\begin{tikzpicture}
\draw[blue] (1,1) .. controls (1.75,2.5) and (1.75,2.5) .. (1,4);
\draw[red] (1.5,1) .. controls (0.75,2.5) and (0.75,2.5) .. (1.5,4);
\draw[blue]  (2,1) .. controls (2.75,2.5) and (2.75,2.5) .. (2,4);
\draw[red]  (2.5,1) .. controls (1.75,2.5) and (1.75,2.5) .. (2.5,4);
\draw[blue]  (3,1) .. controls (3.75,2.5) and (3.75,2.5) .. (3,4);
\draw[red]  (3.5,1) .. controls (2.75,2.5) and (2.75,2.5) .. (3.5,4);
\draw[blue]  (4,1) .. controls (4.75,2.5) and (4.75,2.5) .. (4,4);
\draw[red]  (4.5,1) .. controls (3.75,2.5) and (3.75,2.5) .. (4.5,4);
\draw[blue]  (5,1) .. controls (5.75,2.5) and (5.75,2.5) .. (5,4);
\draw[red]  (5.5,1) .. controls (4.75,2.5) and (4.75,2.5) .. (5.5,4);
\draw[blue]  (6,1) .. controls (6.75,2.5) and (6.75,2.5) .. (6,4);
\draw[red]  (6.5,1) .. controls (5.75,2.5) and (5.75,2.5) .. (6.5,4);
\draw[blue]  (7,1) .. controls (7.75,2.5) and (7.75,2.5) .. (7,4);
\draw[red]  (7.5,1) .. controls (6.75,2.5) and (6.75,2.5) .. (7.5,4);
\draw[blue]  (8,1) .. controls (8.75,2.5) and (8.75,2.5) .. (8,4);
\draw[red]  (8.5,1) .. controls (7.75,2.5) and (7.75,2.5) .. (8.5,4);
\draw (0,-1) -- (9.9,-1);
\draw[->] (4.75,0.75) -- (4.75,-0.5);
\draw[red]  (0.0,3) -- (1.3,3);
\draw[red]  (1.6,3) -- (2.3,3);
\draw[red]  (2.6,3) -- (3.3,3);
\draw[red]  (3.6,3) -- (4.3,3);
\draw[red]  (4.6,3) -- (5.3,3);
\draw[red]  (5.6,3) -- (6.3,3);
\draw[red]  (6.6,3) -- (7.3,3);
\draw[red]  (7.6,3) -- (8.3,3);
\draw[red]  (8.6,3) -- (9.9,3);
\draw[red] (9.5,3.3) node{\mbox{\small$\sigma_\mathcal{X}$}};
\draw[blue]  (0.0,2) -- (0.9,2);
\draw[blue]  (1.2,2) -- (1.9,2);
\draw[blue]  (2.2,2) -- (2.9,2);
\draw[blue]  (3.2,2) -- (3.9,2);
\draw[blue]  (4.2,2) -- (4.9,2);
\draw[blue]  (5.2,2) -- (5.9,2);
\draw[blue]  (6.2,2) -- (6.9,2);
\draw[blue]  (7.2,2) -- (7.9,2);
\draw[blue]  (8.2,2) -- (9.9,2);
\draw[blue] (9.5,1.7) node{\mbox{\small$\tau_\mathcal{X}$}};
\draw[blue] (1.25,2.25) node{\mbox{\tiny$F_{11}$}};
\draw[red] (1.25,2.75) node{\mbox{\tiny$F_{10}$}};
\draw[blue] (2.25,2.25) node{\mbox{\tiny$F_{21}$}};
\draw[red] (2.25,2.75) node{\mbox{\tiny$F_{20}$}};
\draw[blue] (3.25,2.25) node{\mbox{\tiny$F_{31}$}};
\draw[red] (3.25,2.75) node{\mbox{\tiny$F_{30}$}};
\draw[blue] (4.25,2.25) node{\mbox{\tiny$F_{41}$}};
\draw[red] (4.25,2.75) node{\mbox{\tiny$F_{40}$}};
\draw[blue] (5.25,2.25) node{\mbox{\tiny$F_{51}$}};
\draw[red] (5.25,2.75) node{\mbox{\tiny$F_{50}$}};
\draw[blue] (6.25,2.25) node{\mbox{\tiny$F_{61}$}};
\draw[red] (6.25,2.75) node{\mbox{\tiny$F_{60}$}};
\draw[blue] (7.25,2.25) node{\mbox{\tiny$F_{71}$}};
\draw[red] (7.25,2.75) node{\mbox{\tiny$F_{70}$}};
\draw[blue] (8.25,2.25) node{\mbox{\tiny$F_{81}$}};
\draw[red] (8.25,2.75) node{\mbox{\tiny$F_{80}$}};
\draw (9.4,-0.7) node{\mbox{\small$\mathbb{P}(s, t)$}};
\end{tikzpicture}
\caption{Fibration $\pi: \mathcal{X} \to \mathbb{P}(s, t)$ with 8 fibers of type $I_2$.}
\label{fig:EFS}
\end{figure}
\par The results of Proposition~\ref{prop6} and Figure~\ref{diag:isogeny} also hold for suitable lattice extensions induced by the existence of a second involution:
\begin{corollary}
\label{cor:6dim_vGS}
\leavevmode
\begin{enumerate}
\item The 6-dimensional moduli space $\mathfrak{M}_{\langle 2 \rangle \oplus \langle -2 \rangle\oplus D_4(-1)^{\oplus 3}} \subset \mathfrak{M}_{H \oplus N}$ is given by the K3 surfaces admitting commuting van~Geemen-Sarti involutions.
\item The 6-dimensional moduli space $\mathfrak{M}_{H \oplus N_0(-1)}  \subset \mathfrak{M}_{H \oplus N}$ (with the Gram matrix of $N_0$ given in Equation~(\ref{eqn:Gram_matrix_2})) is given by the K3 surfaces admitting a canonical van~Geemen-Sarti involution and a commuting antisymplectic involution induced by an involution on the base curve.
\end{enumerate}
\end{corollary}
\begin{proof}
Part (1) is proved as follows: the existence of a second van~Geemen-Sarti involution means that there is a second 2-torsion section. Its existence implies that there is a homogeneous polynomial $C(s, t)$ of degree four such that $A^2- 4B = C^2$.  One checks that then there are now 12 singular fibers of type $I_2$ on $\mathcal{X}$ and the Mordell-Weil group is $(\mathbb{Z}/2\mathbb{Z})^2$. It follows from results in \cite{Clingher:2021}  that the family of K3 surfaces is polarized by the lattice $\langle 2 \rangle \oplus \langle -2 \rangle\oplus D_4(-1)^{\oplus 3}$. For (2) we can assume that the antisymplectic involution $k_\mathcal{X}$ is induced by the deck transformation $[s:t] \mapsto [-s:t]$ on $\pi_\mathcal{X} : \mathcal{X} \to \mathbb{P}^1\cong \mathbb{P}(s, t)$. A Nikulin involution $\jmath_\mathcal{X}=k_\mathcal{X} \circ (-1)$ is then given by the composition with the (antisymplectic) hyperelliptic involution which acts as $(-1)$ on the fibers.  In turn, this implies that the set of singular fibers of type $I_2$ and $I_1$ on $\mathcal{X}$, respectively, is invariant under $\jmath_\mathcal{X}$, and we obtain the family K3 surfaces in Proposition~\ref{prop4}. 
\end{proof}
\subsection{Relation to mirror symmetry}
\label{ssec:mirror}
 We argued in Section~\ref{sec:CHL} that the \emph{moduli space of the CHL string} is the moduli space $\mathfrak{M}_{H \oplus E_8(-2)}$. We also derived a normal form for these K3 surfaces relating them to certain rational elliptic surfaces. Dolgachev proved that $L = H(2) \oplus E_8(-2)$ diagonally embeds into $\Lambda_\mathrm{K3} = H^3 \oplus E_8(-1)^{\oplus 2}$ such that $L^\perp = H \oplus H(2) \oplus E_8(-2)$ and a splitting $L^\perp = H \oplus L$  is admissible  \cite{MR3586505}. Thus, the moduli space $\mathfrak{M}_{H(2) \oplus E_8(-2)}$ is a self-mirror with respect to this splitting. However, a second admissible splitting is given by $L^\perp = H(2) \oplus L'$ with $L'=H \oplus E_8(-2)$ and yields $\mathfrak{M}_{H \oplus E_8(-2)}$ as the mirror moduli space. 
 \par As we have seen, a polarization by the lattice $H \oplus N$ is equivalent with the existence of a canonical van~Geemen-Sarti involution $\jmath_{\mathcal{X}}: \mathcal{X} \rightarrow \mathcal{X} $ on the K3 surface $\mathcal{X}$. The construction  in Figure~\ref{diag:isogeny} induces an involution on the moduli space $\imath_\mathrm{VGS}: \mathfrak{M}_{H \oplus N} \to \mathfrak{M}_{H \oplus N}$.In the situation of Equation~(\ref{eqn:lattice}) one checks that for $L=H \oplus N$ and $L^\perp = H^2 \oplus N$ the splitting $L^\perp = H \oplus \check{L}$ with $\check{L} \cong L$ is admissible in the sense of Dolgachev's mirror symmetry. Thus, the moduli space of the K3 surfaces given by Equation~(\ref{eqn:X}) is a self-mirror, i.e., $\mathfrak{M}_L \cong \mathfrak{M}_{\check{L}}$ and the van~Geemen-Sarti-Nikulin duality acts as involution on this moduli space. In fact, the action of $\imath_\mathrm{VGS}$ can be identified with the action of mirror symmetry on the corresponding F-theory models.   
\section{The geometric description of the duality}
\label{sec:geometry}
In this section we will prove that the F-theory/CHL string duality can be understood as a relation between certain moduli spaces of lattice polarized K3 surfaces. The key observation is the following isometry of lattices 
\par For the K3 surface $\mathcal{X}$ in Equation~(\ref{eqn:X}) we group the base points of the 8 fibers of type $I_2$ into two unordered sets of 4 elements. Concretely, we group the fibers of type $I_2$ over $B=0$ into two sets by writing $B(s, t)=C(s, t) D(s ,t)$ where $C, D$ are homogenous polynomials of degree four.  The symplectic transformation given by
\beq
\label{eqn:base_change}
 \Big(X, Y, Z \Big) \ \mapsto \  \Big( U, V, W \Big)  = \Big( C(s, t) XZ, \, Z^2, \, C(s, t) Y\Big) \,,
\eeq
yields the equation
\beq
\label{eqn:F}
  \mathcal{F}: \quad W^2 = U V  \,  \Big( C(s, t) \, U^2 - A(s, t) \, UV + D(s, t) \, V^2 \Big) \,.
\eeq 
Equation~(\ref{eqn:F}) puts the (marked) K3 surface into the equivalent form of a \emph{double-quadrics} surface, i.e., a double cover of the Hirzebruch surface $\mathbb{F}_0= \mathbb{P}^1 \times \mathbb{P}^1 \ni ([s:t], [U: V])$ branched over a curve of bi-degree $(4,4)$.  The ruling given by the projection onto the first factor, denoted by $\pi_\mathcal{F}: \mathcal{F} \to \mathbb{P}^1 = \mathbb{P}(s, t)$,  recovers the Jacobian elliptic fibration  in Equation~(\ref{eqn:X}). 
\par Changing the ruling one recovers another elliptic fibration $\pi'_\mathcal{F}: \mathcal{F} \to \mathbb{P}^1 = \mathbb{P}(U, V)$ without section. Concretely, the elliptic fibration $\pi'_\mathcal{F}$ is obtained by re-writing Equation~(\ref{eqn:F}) in the form
\beq
\label{eqn:F2}
   \mathcal{F}:  \quad W^2 =  U V \,  \Big( a_4 (U, V) \, s^4 +   a_3 (U, V) \, s^3 t  + \dots +  a_0 (U, V)\, t^4 \Big) \,,
\eeq 
where $a_i (U, V)$ for $ 0 \le i \le 4$ are homogeneous polynomial of degree two, obtained by requiring that we have
\beq
\label{eqn:hermite00}
  C(s, t) \, U^2 - A(s, t) \, U V+ D(s, t) \, V^2 =  \ a_4 ( U, V ) \, s^4 +   a_3 ( U, V ) \, s^3 t  + \dots +  a_0 ( U, V )\, t^4
\eeq
for all $s, t, U, V$. The relative Jacobian (fibration) $\mathrm{Jac}^0(\pi'_\mathcal{F})$ is the Jacobian elliptic K3 surface $\pi_{\widetilde{\mathcal{Y}}}: \widetilde{\mathcal{Y}} \to \mathbb{P}^1= \mathbb{P}(U, V)$ with the Weierstrass model
\beq
\label{eqn2:Yprime}
 \widetilde{\mathcal{Y}}: \quad y^2 z = x^3 + U^2 V^2  f( U, V ) \, x z^2 +  U^3 V^3  g( U, V ) \, z^3 \,,
\eeq 
where $f, g$ are the homogeneous polynomials of degree four and six, respectively, given by
\beq
\label{eqn:hermite0}
 f = - 4 a_0 a_4 + a_1 a_3 - \frac{1}{3} a_2^2 \,, \qquad g = -\frac{8}{3} a_0 a_2 a_4 + a_0 a_3^2 + a_1^2 a_4 - \frac{1}{3} a_1 a_2 a_3 + \frac{2}{27} a_2^3 \,.
\eeq
For a review and details of these classical formulas; see \cites{MR2166182, MR3995925, MR717601}.    The very general Jacobian elliptic K3 surface $\widetilde{\mathcal{Y}}$ in Equation~(\ref{eqn2:Yprime}) has two singular fibers of type $I_0^*$, twelve singular fibers of type $I_1$, and a trivial Mordell-Weil group. 
\par  We also construct a K3 surface $\mathcal{G}$ as the double cover of $\mathcal{X}$ branched over the even eight that consists of eight components in the reducible fibers over the zeros of $B(s, t)=C(s, t)D(s, t)=0$.  Using the same notation as the one used in Figure~\ref{fig:EFS}, we denote by $F_{i0}$ and $F_{i1}$ for $i=1, \dots, 4$ the components over $C(s, t)=0$ and by $\mathit{F}_{i0}$ and $\mathit{F}_{i1}$ for $i=5, \dots, 8$ the components over $D(s, t)=0$. The even eight we choose this time is different from the one in~(\ref{eq:even_eight}) used in the construction of $\mathcal{X}'$ in Equation~(\ref{eqn:Xprime}).  Over the zeros of $D(s, t)=0$ we choose the components $\{ \mathit{F}_{11}, \dots, \mathit{F}_{41} \}$ of the reducible fibers that are not met by the zero section $\sigma_\mathcal{X}$. Over the zeros of $C(s, t)=0$ we choose  the components $\{ \mathit{F}_{50}, \dots, \mathit{F}_{80} \}$ that \emph{are} met by the zero section $\sigma_\mathcal{X}$. In this way, the double cover branched on $\mathit{F}_{11} + \cdots + \mathit{F}_{41} + \mathit{F}_{50} + \cdots + \mathit{F}_{80}$ is an elliptic fibration $\pi_\mathcal{G}: \mathcal{G} \to \mathbb{P}^1=\mathbb{P}(s ,t)$ with the same singular fibers as $\mathcal{X}'$, but it does not admit a section since both $\sigma_\mathcal{X}$ and $\tau_\mathcal{X}$ now intersect the branch locus. Thus, the total spaces of $\mathcal{G}$ and $\mathcal{X}'$ are not isomorphic, but the relative Jacobian satisfies $\mathrm{Jac}^0(\pi_\mathcal{G}) = \pi_{\mathcal{X}'}$.\footnote{Here and in the following, several elliptic fibrations are stated to admit no sections, which holds only for the very general K3 surfaces. When the relative Jacobian of this fibration is then considered, we do so by direct computation using the general formulae, displayed as~(\ref{eqn:hermite0}) applied to~(\ref{eqn:hermite00}).}  Concretely, the new K3 surface is obtained as the minimal resolution of the double-quadrics  surface
\beq
\label{eqn:G}
   \mathcal{G}: \quad w^2 =   C(s, t) \, u^4 - A(s, t) \, u^2 v^2 + D(s, t) \, v^4 \,,
\eeq
which is branched over a curve bi-degree $(4,4)$ (which has genus 9) in $\mathbb{F}_0 = \mathbb{P}(s, t) \times \mathbb{P}(u, v)$.  Moreover, acting upon the even eight $\mathit{F}_{11} + \cdots + \mathit{F}_{41} + \mathit{F}_{50} + \cdots + \mathit{F}_{80} \in 2 \mathrm{NS}(\mathcal{X})$ with the van Geemen-Sarti involution on $\mathcal{X}$ yields the even eight $\mathit{F}_{10} + \dots + \mathit{F}_{40} + \mathit{F}_{51} + \dots + \mathit{F}_{81}$. The double cover branched on the latter yields a genus-one fibration isomorphic to Equation~(\ref{eqn:G}), as it simply corresponds to the interchange $u \leftrightarrow v$.  We call the two combinations of eight exceptional curves selected from the components of the reducible fibers, namely $\{ \mathit{F}_{10}, \dots, \mathit{F}_{40}, \mathit{F}_{51}, \dots, \mathit{F}_{81}\}$ and $\{ \mathit{F}_{11}, \dots, \mathit{F}_{41}, \mathit{F}_{50}, \dots, \mathit{F}_{80}\}$, whose induced double covers yield isomorphic genus-one fibrations a \emph{marked pair of even eights} on $\mathcal{X}$. 
\par A change of ruling yields a second elliptic fibration $\pi'_\mathcal{G}: \mathcal{G} \to \mathbb{P}^1=\mathbb{P}(u, v)$ given by
\beq
\label{eqn:G_b}
    \mathcal{G}: \quad  w^2 =  a_4 (u^2, v^2) \, s^4 +   a_3 (u^2, v^2) \, s^3 t  + \dots +  a_0 (u^2, v^2)\, t^4  \,,
\eeq 
where we assumed Equation~(\ref{eqn:hermite00}). It follows that the relative Jacobian $\mathrm{Jac}^0(\pi'_\mathcal{G})$ is the Jacobian elliptic K3 surface $\pi_\mathcal{Y}: \mathcal{Y} \to \mathbb{P}^1= \mathbb{P}(u, v)$ with the Weierstrass model
\beq
\label{eqn2:Y}
 \mathcal{Y}: \quad Y^2 Z = X^3 +   f( u^2, v^2) \, X Z^2 +  g( u^2, v^2) \, Z^3 \,,
\eeq 
where the polynomials $f, g$ are the same polynomials as in Equation~(\ref{eqn:hermite0}). Thus, the K3 surfaces in Equation~(\ref{eqn2:Y}) and~(\ref{eqn2:Yprime}) coincides with K3 surfaces in Equation~(\ref{eqn:Y}) and~(\ref{eqn:Yprime}), respectively, obtained from the general rational elliptic surface $\mathcal{R}$ in Equation~(\ref{eqn:R}),  normalized so that $d_0=0$ and $d_\infty=0$. We have the following:
\begin{lemma}
\label{lem:sections}
The elliptic fibration $\pi'_\mathcal{G}: \mathcal{G} \to \mathbb{P}(u, v)$ admits sections.
\end{lemma}
\begin{proof}
The double cover $\mathcal{G} \to \mathcal{X}$ is constructed using an even eight that consists of 4 neutral components and 4 non-neutral components. It follows that the double cover has an elliptic fibration with respect to the projection onto $\mathbb{P}(s, t)$, but the sections of $\mathcal{X}$ do not lift to sections as they intersect the branch locus. The preimages of the 8 fibers of type $I_1$ then determine in the very general case 16 sections of the fibration with respect to the projection onto $\mathbb{P}(u, v)$. The sections can be computed explicitly: they are the solutions given by the 8 fixed constants $[s:t]$  obtained from solving $A(s,  t)^2 - 4 C(s, t) \, D(s, t) =0$ and then solving for $w$ in the equation $4 C(s, t) w^2 = (2 C(s, t) u^2 - A(s, t) v^2)^2$.
\end{proof}
\par One can also construct a third kind of double covers from $\mathcal{X}$ by grouping the eight fibers of type $I_2$ over $B=0$ into sets of two and six elements:
\begin{remark}
If we group the eight fibers of type $I_2$ over $B=0$ into sets of 2 and 6 elements by writing $B(s, t)=C'(s, t) D'(s ,t)$ where $C', D'$ are homogenous polynomials of degree 2 and 6, respectively, we obtain a K3 double cover as the minimal resolution of the double-sextic surface
\beq
\label{eqn23Yprime}
    w^2 =   C'(s, t) \, u^4 - A(s, t) \, u^2 + D'(s, t)  \,,
\eeq
which is branched over the strict transform of a sextic in $\mathbb{F}_1 = \mathbb{P}^2 = \mathbb{P}(s, t, u)$. 
\end{remark}
Returning to the K3 surface $\mathcal{G}$, note that the Jacobian elliptic fibration on $\mathcal{Y}$ was obtained by taking the relative Jacobian of the elliptic fibration on $\mathcal{G}$ induced by the projection onto $\mathbb{P}(u, v)$. However, the existence of a section implies that the construction of the relative Jacobian is realized as a birational map. This is a classical result by Hermite and explicit formulas are given in Section~\ref{ssec:AJM}. It follows that K3 surfaces $\mathcal{G}$ and $\mathcal{Y}$ are isomorphic and carry the same lattice polarization $H \oplus E_8(-2)$.  We then have the following:\footnote{For a definition of the multi-section index of a genus-one fibration; see \cite{MR986969}.}
\begin{proposition}
\label{prop:G}
Let $\mathcal{G}$ be a K3 surface, as in Equation~(4.7). Then, in the very general case, the Neron-Severi lattice $\mathrm{NS}(\mathcal{G})$ is $H \oplus E_8(-2)$.  In particular, the elliptic fibration $\pi_\mathcal{G}: \mathcal{G} \to \mathbb{P}(s, t)$ admits a bi-section and the elliptic fibration $\pi'_\mathcal{G}: \mathcal{G} \to \mathbb{P}(u, v)$ admits a section.
\end{proposition}
\begin{proof}
Lemma~\ref{lem:sections} shows that for $\pi'_\mathcal{G}$ one has sections whence $\mathcal{G} \cong \mathcal{Y}$. For $\pi_\mathcal{G}$ there clearly exist bisections. Moreover,  the N\'eron-Severi lattice $\operatorname{NS}(\mathcal{G})$  is a sublattice of the lattice $\operatorname{NS}(\mathcal{X}')$ of index $l_1$ such that  $\det{q_{\operatorname{NS}(\mathcal{G})}}= l_1^2 \det{q_{\operatorname{NS}(\mathcal{X}')}}$ with $l_1=2$.  The number $l_1$ is the multi-section index of fibration~(\ref{eqn:G}); see \cite{MR1707196}*{Lemma 2.1}.
\end{proof}
Note that the defining Jacobian elliptic fibration on $\mathcal{Y}$ may be identified with the relative Jacobian of the elliptic fibration on $\mathcal{G}$ induced by the projection onto $\mathbb{P}(u, v)$. The presence of sections on the latter implies, however, that the relative Jacobian map is birational. This is a classical result by Hermite and explicit formulas realizing the identification are given in Section~\ref{ssec:AJM}. The two K3 surfaces $\mathcal{G}$ and $\mathcal{Y}$ are isomorphic and carry the same  Neron-Severi lattice:  $H \oplus E_8(-2)$. 
\par The decomposition $H \oplus E_8(-2)$ in $\mathrm{NS}(\mathcal{G})$ may be described explicitly. To see this, note that the double cover  $\mathcal{G} \to \mathcal{X}$ is branched over a specific even-eight curve configuration. This configuration consists of eight rational curves on $\mathcal{X}$, four of which are neutral components of four $I_2$ fibers and additional four are non-neutral components of the remaining $I_2$ fibers; see Section~\ref{sec:VGS_K3}. Furthermore, the preimages of the section and 2-torsion section if the Jacobian elliptic fibration on $\mathcal{X}$ give two disjoint elliptic curves on $\mathcal{G}$ that are smooth members of an elliptic pencil $\vert F \vert $. In the context of this elliptic fibration, one has eight pairs of sections $L_i,R_i$, $1\leq i \leq 8$, obtained as preimages of the eight $I_1$ fibers on $\mathcal{X}$. Consider $R_1$ as the zero-section. Then $\langle R_1, F \rangle$ is isomorphic to $H$.  Let us examine the orthogonal complement $\langle R_1, F \rangle ^{\perp} $ in $\mathrm{NS}(\mathcal{G})$. Denote $ B_1  = L_1 - R_1 - 4 F$ and $ B_i = L_i - R_i - 2F $, $ 2 \leq i \leq 8$.  One obtains that $(1/2) \sum B_i  \in \langle R_1, F \rangle ^{\perp} $ and furthermore $\langle R_1, F \rangle ^{\perp}  = \langle  B_1, B_2, \dots, B_8, (1/2) \sum B_i \rangle $. Upon checking the intersection matrix, one obtains that $\langle R_1, F \rangle ^{\perp} = Q(-2)$ where $Q$ is a unimodular, even positive-definite lattice of rank eight.  It follows that $\langle R_1, F \rangle ^{\perp} $ is isomorphic to $E_8(-2)$.
\begin{remark}
The situation of Proposition~\ref{prop:G} is depicted in Figure~\ref{fig:FTH-CHL}.  The rational maps $\mathcal{F} \to \widetilde{\mathcal{Y}}$ and $\mathcal{G} \to \mathcal{X}'$ are reflected by the lattice-theoretic identities
$H \oplus N \cong H(2) \oplus D_4(-1)^{\oplus 2}$ and $H(2) \oplus N \cong H \oplus E_8(-2)$, respectively.
\end{remark}
\begin{figure}
\beqn
\xymatrix{
 \mathcal{X}'  \ar@<-2pt>[d]  &  \mathcal{G} \ar[l]_{\mathrm{Jac}^0} \ar[dl]  \ar[d]  \ar[r]^{\cong} & \mathcal{Y}  \ar[d] \ar[dr]\\
 \mathcal{X}  \ar@<-2pt>[u]  &  \mathcal{F} \ar[l]_{\cong}   \ar[r]^{\mathrm{Jac}^0} & \widetilde{\mathcal{Y}} &  \mathcal{R}\\
 }
\eeqn
\caption{Geometric construction of the F-theory/CHL string duality}
\label{fig:FTH-CHL}
\end{figure}
\par As explained before, the K3 surface $\mathcal{F}$ was obtained as double covers of $\mathbb{F}_0$ branched over the vanishing divisor of a section in the line bundle $\mathscr{L}=\mathcal{O}_{\mathbb{F}_0}(4,4)$. This  branch locus is reducible: it decomposes as $\mathit{L}_0 \cup \mathit{L}_\infty  \cup \mathit{E}$ into two lines $\mathit{L}_0, \mathit{L}_\infty$ and a curve $\mathit{C}$. The latter is a curve of genus three in the linear system $|\mathscr{L}|$.  Let $\mathfrak{W}$ be the moduli space of K3 surfaces that can be obtained as the minimal resolution of double-quadrics over $\mathbb{F}_0$ whose branch locus is the union of divisors of type $(1,0)$, $(1,0)$, and $(2,4)$.  We note that the moduli space $\mathfrak{W}$ is precisely the moduli space of F-theory compactifications dual to the CHL string considered by Witten in \cite{MR1615617}: each general element $\mathcal{F} \in \mathfrak{W}$ admits the Jacobian elliptic fibration $\pi_\mathcal{F}$ and the elliptic fibration $\pi'_\mathcal{F}$ without section, corresponding to the $\Gamma_0(2)$ elliptic fibration and the inherited elliptic fibration, respectively. Similarly, let $\mathfrak{W}'$ be the moduli space of K3 surfaces that are obtained as the minimal resolution of double-quadrics over $\mathbb{F}_0$ whose branch locus is of type $(4,4)$ and invariant under the involution $([u:v], [s:t]) \to ([u:-v], [s:t])$. The moduli space $\mathfrak{W}'$ is the moduli space of  CHL strings dual to F-theory compactifications considered in \cite{MR1797021}. We have the following:
\begin{proposition}
\label{prop:F}
$\mathfrak{W}'$ is a finite covering spaces of $\mathfrak{M}_{H \oplus N}$.
\end{proposition}
\begin{proof}
Every marked pair of even eights on a K3 surface $\mathcal{X} \in \mathfrak{M}_{H \oplus N}$ allows to construct a K3 surface $\mathcal{G} \in \mathfrak{W}'$ using Equation~(\ref{eqn:G}). Conversely, for every K3 surface $\mathcal{G} \in \mathfrak{W}'$, Equation~(\ref{eqn:X}) then determines a double cover of the Hirzebruch surface $\mathbb{F}_4$ and an element in $\mathfrak{M}_{H \oplus N}$. The K3 surfaces $\mathcal{G} \in \mathfrak{W}'$ obtained as the double coverings of an element in $\mathfrak{M}_{H \oplus N}$ branched along all possible marked pairs of even eights all project on the same element $\mathcal{X}$.
\end{proof}
\par We also have the equation for a two-parameter family of double-quadrics  surfaces given by
\beq
\label{eqn:F_2param}
  \mathcal{F}_{(d_0, d_\infty)}: \quad W^2 = \big(U - d_0 V\big)  \big(d_\infty U -  V\big) \ \Big( C(s, t) \, U^2 - A(s, t) \, UV + D(s, t) \, V^2 \Big) \,,
\eeq
where the two fibers of type $I_0^*$ are now at $[U:V]=[d_0:1]$ and $[1:d_\infty]$. A base transform, similar to Equation~(\ref{eqn:base_change}), converts Equation~(\ref{eqn:F_2param}) into a double cover of the Hirzebruch surface $\mathbb{F}_4$ by forgetting the marking, given by
\beq
\label{eqn:X_2param}
\begin{split}
 \mathcal{X}_{(d_0, d_\infty)}: \quad Y^2 Z =  X^3  & -  \big(2 d_0 C+ 2 d_\infty D -(1+ d_0 d_\infty) A \big)  \, X^2Z \\
 & +  \big(C + d_\infty^2 D - d_\infty A\big)\big(d_0^2 C + D - d_0 A\big)   X Z^2 ,
\end{split}
\eeq 
with the discriminant
\beq
\label{eqn:DiscrX_2param}
  \Delta_{\mathcal{X}_{(d_0, d_\infty)}} = (d_0 d_\infty -1)^2  \big(C + d_\infty^2 D - d_\infty A\big)^2 \big(d_0^2 C + D - d_0 A\big)^2 \big(A^2 - 4 C D\big)\,.
\eeq
A comparison with the discriminant of $\mathcal{X}$ in Equation~(\ref{eqn:DiscrX_Xp}) shows that only the location of the eight fibers of type $I_2$ has changed whereas the location of the fibers of type $I_1$ has remained fixed.  Moreover, for $d_0 = d_\infty =0$ the surfaces $\mathcal{F}_{(d_0, d_\infty)}$ and $\mathcal{F}$ coincide. It turns out that the relation between $\mathcal{F}$ and $\mathcal{F}_{(d_0, d_\infty)}$ is symmetric. That is, replacing polynomials as follows
\beq
\label{eqn:involution}
 \imath_{(d_0, d_\infty)}: \quad 
\left(\begin{array}{c} A \\ C \\ D \end{array}\right) \ \mapsto \ \left(\begin{array}{c} A' \\ C' \\ D' \end{array}\right)
= \left(\begin{array}{c} 2 d_0 C+ 2 d_\infty D -(1+ d_0 d_\infty) A  \\  C + d_\infty^2 D - d_\infty A \\ d_0^2 C + D - d_0 A \end{array}\right) \,,
\eeq
maps $\mathcal{F}$ in Equation~(\ref{eqn:F}) to $ \mathcal{F}_{(d_0, d_\infty)}$ and vice versa (up to a trivial rescaling of the equation defined over $\mathbb{Q}[d_0d_\infty-1]$). Thus, for $d_0 d_\infty \not = 1$ Equation~(\ref{eqn:involution}) defines a two-parameter family of involutions $\imath_{(d_0, d_\infty)}$ on  the moduli space $\mathfrak{W}$.
\subsection{Connection to the Brauer group}
As explained above, the N\'eron-Severi  lattice $\operatorname{NS}(\mathcal{F}) \cong H \oplus N \cong H(2) \oplus D_4(-1)^{\oplus 2}$ is a sublattice of the lattice $\operatorname{NS}(\widetilde{\mathcal{Y}})= H \oplus D_4(-1)^{\oplus 2}$ of index two.  We now consider the problem of reconstructing $\mathcal{F}$ from $\widetilde{\mathcal{Y}}$, that is, constructing a polarized Hodge substructure of the transcendental lattice $\mathrm{T}_{\widetilde{\mathcal{Y}}}$ that is the transcendental lattice of another K3 surface. One checks that
\beq
 \Gamma:=\mathrm{T}_{\widetilde{\mathcal{Y}}} =  H \oplus H \oplus D_4(-1)^{\oplus 2} \,.
\eeq
Recall that an element $\theta$ of order $n$ in the \emph{Brauer group} $\operatorname{Br}(\widetilde{\mathcal{Y}}) =\operatorname{Hom}(\mathrm{T}_{\widetilde{\mathcal{Y}}}, \mathbb{Q}/\mathbb{Z})$ is a surjective homomorphism $\theta: \mathrm{T}_\mathcal{Y}  \to \mathbb{Z}/n\mathbb{Z}$ and defines a sublattice of index $n$ of $\Gamma$. This sublattice is given by $\mathrm{T}_{\langle \theta \rangle} = \ker{( \theta: \Gamma \to  \mathbb{Z}/n\mathbb{Z})}$ where $\langle \theta \rangle$ denotes  the cyclic subgroup generated by $\theta$. Conversely, any sublattice $\Gamma'$ of index $n$ in $\Gamma$ with cyclic quotient group $\Gamma/\Gamma'$ determines a subgroup of order $n$ in $\operatorname{Br}(\widetilde{\mathcal{Y}})$.  If there exists a primitive lattice embedding of $\mathrm{T}_{\langle \theta \rangle}$ into the K3 lattice $\Lambda_\mathrm{K3}$,  then the Hodge structure $\mathrm{T}_{\langle \theta \rangle}$ is guaranteed to be the transcendental lattice of another K3 surface $\mathcal{F}$.  Since the lattice embedding is in general not unique, neither is the surface $\mathcal{F}$; two K3 surfaces with isomorphic transcendental lattice $\mathrm{T}_{\langle \theta \rangle}$ are Fourier-Mukai partners. 
\par In our situation, we are considering elements $\theta \in \operatorname{Br}(\widetilde{\mathcal{Y}})_2$ of order two such that the sublattice $\mathrm{T}_{\langle \theta \rangle} = \ker{( \theta: \Gamma \to  \mathbb{Z}/2\mathbb{Z})}$ has index two in $\Gamma$.  The existence of a primitive lattice embedding $\mathrm{T}_{\langle \theta \rangle} \hookrightarrow \Lambda_\mathrm{K3}$ is a-priori guaranteed since the K3 surface $\widetilde{\mathcal{Y}}$ already admits an elliptic fibration with a section; see \cite{MR2166182}.  
\par Using the arguments of \cite{MR2166182} it follows that there are isomorphism classes for elements in the Brauer group $\operatorname{Br}(\widetilde{\mathcal{Y}})_2$, with representatives $\Gamma_{2,c}=\Lambda_{2, c} \oplus H \oplus D_4(-1)^{\oplus 2}$ for $c=0, 1$.  Here, we are using the notation  $\Lambda_{b, c}=(0,b,2c)$ for the indefinite lattice of rank two such that $\Lambda_{2,0} = H(2)$ and $\Lambda_{2,1} = \langle 2 \rangle \oplus \langle -2 \rangle$. Notice that there are more sublattices of $\Gamma$ of index 2. However, we are only considering elements with the isomorphism classes $\Gamma_{2,c}$. One can then construct projective models for the corresponding K3 surfaces $\mathcal{F}_{2,c}$ with the N\'eron-Severi lattice $\Lambda_{2, c}  \oplus D_4(-1)^{\oplus 2}$.  The construction of the projective models is based on the well-developed theory of linear series on K3 surfaces; see \cite{MR364263}.  It follows that the a projective model for the K3 surface $\mathcal{F}_{2,0}$ is always that of a double cover of $\mathbb{F}_0$ branched over a curve of bi-degree $(4,4)$, necessarily with two elliptic fibrations, $\pi_\mathcal{F}$ and $\pi'_\mathcal{F}$, associated with the two elliptic pencils; see \cite{MR2166182}*{Sec.~5.5}. Similarly, $\mathcal{F}_{2,1}$ is obtained as the desingularization of a double-sextic surface; see \cite{MR2166182}*{Sec.~5.6}. 
\par From Ogg-Shafarevich theory it follows that $\mathcal{F}_{2,c}$ with $c=0, 1$ admits a genus-one fibration $\pi'_{\mathcal{F}_{2,c}} : \mathcal{F}_{2,c} \to \mathbb{P}^1$ with a bisection such that the relative Jacobian fibration recovers $\pi_\mathcal{Y}: \mathcal{Y} \to \mathbb{P}^1$. (For $\mathcal{F}_{2,0}$ we always take this elliptic fibration to coincide with the elliptic fibration $\pi'_\mathcal{F}$ introduced before.) Moreover,  every pair of elements $\pm \theta \in \operatorname{Br}(\widetilde{\mathcal{Y}})_2$  uniquely determines such a genus-one fibration on $\mathcal{F}_{2,c}$; see \cite{MR2166182}*{Sec.~4.1}.  
We then have the following:
\begin{proposition}
\label{prop:coverging2}
$\mathfrak{W}$ is a finite covering space of $\mathfrak{M}_{H \oplus D_4(-1)^{\oplus 2}}$. In particular, for a K3 surface $\widetilde{\mathcal{Y}} \in \mathfrak{M}_{H \oplus D_4(-1)^{\oplus 2}}$ and every pair $\pm \theta \in \operatorname{Br}(\widetilde{\mathcal{Y}})_2$ in the isomorphism class $\Gamma_{2,0}$, the Hodge substructure $\mathrm{T}_{\langle \theta \rangle}$ matches the one on the transcendental lattice of a K3 surface $\mathcal{F}$ in Equation~(4.2).
\end{proposition}
\begin{proof}
Every pair of elements $\pm \theta \in \operatorname{Br}(\widetilde{\mathcal{Y}})_2$ in the isomorphism class $\Gamma_{2,0}$ determines an equation for $\mathcal{F}_{2,0}$ as a double cover of $\mathbb{F}_0$ branched over a curve of bi-degree $(4,4)$, with a genus-one fibration $\pi'_{\mathcal{F}_{2,0}} : \mathcal{F}_{2,0} \to \mathbb{P}^1=\mathbb{P}(u, v)$ with multi-section index 2 such that the relative Jacobian fibration is $\pi_{\widetilde{\mathcal{Y}}}: \widetilde{\mathcal{Y}} \to \mathbb{P}(u, v)$. Because of the lattice isomorphism $H(2) \oplus D_4(-1)^{\oplus 2} \cong H \oplus N$, the second elliptic fibration $\pi_{\mathcal{F}_{2,0}}$ must admit a section and a 2-torsion section. In fact, it was proved in \cite{Clingher:2021} that there is a unique Jacobian elliptic fibration on the K3 surface with N\'eron-Severi lattice $H \oplus N$, i.e., the Jacobian elliptic fibration with singular fibers $8 I_2 + 8 I_1$ and Mordell-Weil group $\mathbb{Z}/2\mathbb{Z}$. A fractional linear transformation on $\mathbb{P}(U,V)$ then allows us to bring $\mathcal{F}_{2,0}$ in the form of Equation~(\ref{eqn:F}). Conversely, taking the relative Jacobian $\mathrm{Jac}^0(\pi'_{\mathcal{F}_{2,0}})$ we always obtain a Jacobian elliptic K3 surface $\widetilde{\mathcal{Y}} \in \mathfrak{M}_{H \oplus D_4(-1)^{\oplus 2}}$ such that K3 surfaces $\mathcal{F}_{2,0} \in \mathfrak{W}$ obtained from different factorization $B(s, t)=C(s, t)D(s, t)$ yield the same Jacobian elliptic K3 surface $\widetilde{\mathcal{Y}}$.
\end{proof}
We have the following result:
\begin{theorem}
\label{thm:duality}
In Figure~\ref{fig:FTH-CHL} we have the following:
\begin {enumerate}
\item every K3 surface $\mathcal{X} \in \mathfrak{M}_{H \oplus N}$ has an algebraic correspondence with an element $\mathcal{Y} \in \mathfrak{M}_{H \oplus E_8(-2)}$ and vice versa,
\item every K3 surface $\mathcal{X} \in \mathfrak{M}_{H \oplus N}$ has an algebraic  correspondence with an element  $\widetilde{\mathcal{Y}} \in \mathfrak{M}_{H \oplus D_4(-1)^{\oplus 2}}$ and vice versa.
\end{enumerate}
A correspondence is a double-quadric $\mathcal{G} \in \mathfrak{W}'$ in (1) and $\mathcal{F} \in \mathfrak{W}$ in (2).
\end{theorem}
\begin{proof}
The proof follows from Propositions~\ref{prop:G}, \ref{prop:F}, \ref{prop:coverging2}.
\end{proof}
The discrete choices involved in Theorem~\ref{thm:duality} are as follows:
\begin{remark}
One has:
\begin {enumerate}
\item for every $\mathcal{X} \in \mathfrak{M}_{H \oplus N}$, an element $\mathcal{G} \in\mathfrak{M}_{H \oplus E_8(-2)}$ is determined by a marked pair of even eight configurations $\mathit{F}_{1i} + \dots + \mathit{F}_{4i} + \mathit{F}_{5j} + \dots + \mathit{F}_{8j} \in 2 \mathrm{NS}(\mathcal{X})$, for $(i, j)=(0,1), (1,0)$,
\item for every $\widetilde{\mathcal{Y}} \in \mathfrak{M}_{H \oplus D_4(-1)^{\oplus 2}}$, an element $\mathcal{F} \in \mathfrak{M}_{H \oplus N}$  is determined by a pair $\pm \theta \in \operatorname{Br}(\widetilde{\mathcal{Y}})_2$ in the isomorphism class $\Gamma_{2,0}$.
\end{enumerate}
\end{remark}
\section{A natural subspace of dimension six}
\label{sec:specialization}
The construction of the K3 surface $\mathcal{Y}$ in Equation~(\ref{eqn:Y}), considered a CHL string background,  is based on a rational elliptic surface $\mathcal{R}$; see Proposition~\ref{prop1}. We recall that in the general case the existence of an antisymplectic involution $k_\mathcal{Y}$ on $\mathcal{Y}$ (such that the minimal resolution of $\mathcal{Y}/\langle k_\mathcal{Y} \rangle$ is isomorphic to $\mathcal{R}$) is equivalent to $\mathrm{NS}(\mathcal{Y}) \cong H \oplus E_8(-2)$. One can ask what happens when there also exists a 2-torsion section $\tau_\mathcal{R} \in \mathrm{MW}(\mathcal{R}, \pi_\mathcal{R})$ on the rational elliptic surface. If there is such a section, then the  K3 surface $\mathcal{Y}$ will inherit this property. In fact, one easily checks that there exists a 2-torsion section $\tau_\mathcal{R} \in \mathrm{MW}(\mathcal{R}, \pi_\mathcal{R})$ if and only if there exists a 2-torsion section $\tau_\mathcal{Y} \in \mathrm{MW}(\mathcal{Y}, \pi_\mathcal{Y})$. The physical relevance of this requirement for the CHL string was discussed in \cite{MR3995925}. The classification of rational elliptic surfaces in \cite{MR1104782} provides an answer for what this rational elliptic surface $\mathcal{R}$ with a 2-torsion section is: it must satisfy
\beq
\label{eqn:MW_special}
 \mathrm{MW}(\mathcal{R}, \pi_\mathcal{R}) \ \cong \  \mathbb{Z}/2\mathbb{Z} \oplus D_4^\vee .
\eeq
Equation~(\ref{eqn:MW_special}) describes the rational elliptic surface of \emph{highest} rank, admitting a 2-torsion section. Moreover, there are no other rational elliptic surfaces with section and $\mathrm{rank}(\mathrm{MW})\ge 4$ that contain a 2-torsion section. We note that Equation~(\ref{eqn:MW_special})  is precisely the second case considered earlier in~(\ref{eqn:casesRES}). It follows that now there is also a (canonical) van Geemen-Sarti involution on $\mathcal{Y}$, i.e., a symplectic involution obtained as the translation by the 2-torsion section.  Thus, we have the following:
\begin{corollary}
\label{cor:subspace1}
The moduli space of F-theory models with discrete flux admitting an antisymplectic involution induced by an involution on the base curve whose quotient is a rational elliptic surface is naturally isomorphic to the moduli space of the CHL string backgrounds admitting a (symplectic) van Geemen-Sarti involution.
\end{corollary}
For a general CHL string background $\mathcal{Y}$ we also constructed the Jacobian elliptic K3 surface $\widetilde{\mathcal{Y}}$ in Equation~(\ref{eqn:Yprime}) as the minimal resolution of $\mathcal{Y}/ \langle \jmath_\mathcal{Y}\rangle$ using the Nikulin involution $\jmath_\mathcal{Y}=k_\mathcal{Y} \circ (-1)$. The K3 surfaces $\widetilde{\mathcal{Y}}$ form the 10-dimensional moduli space $\mathfrak{M}_{H \oplus D_4(-1)^{\oplus 2}}$; see Proposition~\ref{prop4}.  Theorem~\ref{thm:duality} proves that the algebraic correspondences between the K3 surfaces $\mathcal{X}$ and $\widetilde{\mathcal{Y}}$ form a moduli space $\mathfrak{W}$. Here, $\mathfrak{W}$ is the moduli space of K3 surfaces that can be obtained as the minimal resolution of double-quadrics  surfaces over $\mathbb{F}_0$ whose branch locus is the union of divisors of type $(1,0)$, $(1,0)$, and $(2,4)$. 
\par Proposition~\ref{cor:2coverK3_14} proved that -- as the lattice polarization of $\mathcal{Y}$ extends from $H \oplus E_8(-2)$ to $H \oplus N_0(-1)$ -- the lattice polarization of $\widetilde{\mathcal{Y}}$ extends from $H \oplus D_4(-1)^{\oplus 2}$ to $H \oplus D_8(-1) \oplus  A_1(-1)^{\oplus 4}$. Figure~\ref{fig:FTH-CHL} can then be extended into Figure~\ref{fig:FTH-CHL_2} where now the Jacobian elliptic K3 surfaces $\mathcal{X}, \mathcal{X}'$, considered F-theory backgrounds, also admit antisymplectic involutions (with rational quotient surfaces $\widetilde{\mathcal{R}}, \widetilde{\mathcal{R}}'$ and K3 quotients $\widetilde{\mathcal{X}}, \widetilde{\mathcal{X}}'$) and  the Jacobian elliptic K3 surface $\mathcal{Y}$, considered a CHL background, also admits a van Geemen-Sarti-Nikulin dual $\mathcal{Y}'$ which in turn inherits an antisymplectic involution with rational quotient surface $\mathcal{R}'$. Antisymplectic involutions also lift to the double-quadrics  surfaces $\mathcal{F}$ and $\mathcal{G}$ whose quotients (after composing with the hyperelliptic involution) are $\widetilde{\mathcal{F}}$ and $\widetilde{\mathcal{G}}$, respectively. For the K3 surface $\mathcal{F}$ in Equation~(\ref{eqn:F}) the existence of a compatible antisymplectic involutions implies that the polynomials $A, C, D$ satisfy 
\beq
 A(s, t) = - \alpha(s^2, t^2) \,, \qquad C(s, t) = \gamma(s^2,t^2) \,, \qquad D(s, t) = \delta(s^2, t^2) \,,
\eeq
for some homogeneous polynomials  $\alpha, \gamma, \delta$ of degree two. (Here, we introduced a minus sign  for convenience.) We have the following:
\begin{proposition}
In the situation of Corollary~\ref{cor:subspace1} there exist the algebraic correspondences in  Figure~\ref{fig:FTH-CHL_2}. The  defining equations for the surfaces are given in Table~\ref{tab:WEQ}. Here,  $\alpha, \gamma, \delta$ are general homogeneous polynomials of degree 2, and the polynomials $a, c, d$ of degree 2 are defined by requiring
\beq
\label{eqn:hermite00_2}
  \gamma(S, T) \, U^2 + \alpha(S, T) \, U V+ \delta(S, T) \, V^2 =  \  c( U, V ) \, S^2+   a ( U, V ) \, S T  + d( U, V )\, T^2
\eeq
for all $S, T, U, V$.  In particular, we have:
\begin{enumerate}
\item $\mathcal{X}, \mathcal{X}', \mathcal{Y}, \mathcal{Y}', \mathcal{G}, \mathcal{F}, \widetilde{\mathcal{G}},  \in \mathfrak{M}_{H \oplus N_0(-1)}$,
\item $\widetilde{\mathcal{X}}, \widetilde{\mathcal{X}}', \widetilde{\mathcal{Y}}, \widetilde{\mathcal{Y}}',  \widetilde{\mathcal{F}}, \in \mathfrak{M}_{H \oplus D_8(-1) \oplus  A_1(-1)^{\oplus 4}}$,
\item $\mathcal{R}, \mathcal{R}', \widetilde{\mathcal{R}},   \widetilde{\mathcal{R}}'$ are rational elliptic surfaces with $\mathrm{MW}(\mathcal{R}, \pi_\mathcal{R}) \ \cong \  \mathbb{Z}/2\mathbb{Z} \oplus D_4^\vee$,
\item $\mathcal{G}, \mathcal{F}, \widetilde{\mathcal{G}}, \widetilde{\mathcal{F}}$ are double-quadrics  surfaces over $\mathbb{F}_0$.
\end{enumerate}
\end{proposition}
\begin{proof}
One uses the explicit constructions for double covers in Section~\ref{ssec:double_covers} and Section~\ref{sec:VGS_K3} to construct all surfaces in Figure~\ref{fig:FTH-CHL_2} explicitly. The Gram matrix of $N_0$ was given in Equation~(\ref{eqn:Gram_matrix_2}).  The existence of sections that imply isomorphisms $\mathcal{X}' \cong \mathcal{G} \cong \mathcal{Y}$ and $\widetilde{\mathcal{G}} \cong \mathcal{Y}'$ is proved as in Lemma~\ref{lem:sections}. The existence of sections implies $\widetilde{\mathcal{X}} \cong \widetilde{\mathcal{F}} \cong \widetilde{Y}'$ and follows immediately from the equation for $\widetilde{\mathcal{F}}$ in Table~\ref{tab:WEQ}.
\end{proof}
We make the following:
\begin{remark}
A marking of $\mathcal{X}$, given by the factorization $\gamma(s^2, t^2) \cdot \delta(s^2, t^2)$, determines the double cover $\mathcal{G}$. The marking also induces a canonical marking of $\mathcal{Y}'$, given by the factorization $c(u^2, v^2) \cdot d(u^2, v^2)$, with the same double cover $\mathcal{G}$.
\end{remark}
\begin{remark}
The fundamental object in Figure~\ref{fig:FTH-CHL_2} is the double-quadric $\widetilde{\mathcal{F}}$. All other K3 surfaces can be obtained from $\widetilde{\mathcal{F}}$ using various the double covers.
\end{remark}
\begin{figure}
\beqn
\xymatrix{
&&& \widetilde{\mathcal{G}} \ar@/_1pc/[ddll]_{\mathrm{Jac}^0}   \ar[r]^{\cong} & \mathcal{Y}'  \ar@<-2pt>[d]  \ar[dr] \ar[drr]  \\
&&  \mathcal{X}'  \ar@<-2pt>[d] \ar[dl] \ar[dll]|(.32)\hole &   \mathcal{G} \ar[l]_{\cong} \ar[dl] \ar[ur] \ar[d]  \ar[u] \ar[r]^{\cong} & \mathcal{Y} \ar@<-2pt>[u]   \ar[dr]|(.35)\hole  \ar[drr]|(.25)\hole |(.50)\hole & \widetilde{\mathcal{Y} }' \ar@<-2pt>[d]  & \mathcal{R}'  \ar@<-2pt>[d] \\
\widetilde{\mathcal{R}}'  \ar@<-2pt>[d] 	& \widetilde{\mathcal{X}}'  \ar@<-2pt>[d] 	& \mathcal{X}  \ar@<-2pt>[u]  \ar[dl]  \ar[dll]|(.505)\hole	&   \mathcal{F} \ar[l]_{\cong} \ar[d]   \ar[rr]^{\mathrm{Jac}^0}|(.3)\hole  &&  \widetilde{\mathcal{Y}} \ar@<-2pt>[u]   & \mathcal{R} \ar@<-2pt>[u]  \\
\widetilde{\mathcal{R}} \ar@<-2pt>[u]  	& \widetilde{\mathcal{X}} \ar@<-2pt>[u]  	& & \widetilde{\mathcal{F}} \ar[ll]_{\cong}  \ar@/^1pc/[uurr]^{\cong}  \\
 }
\eeqn
\caption{Extension of the F-theory/CHL string duality (Case I)}
\label{fig:FTH-CHL_2}
\end{figure}
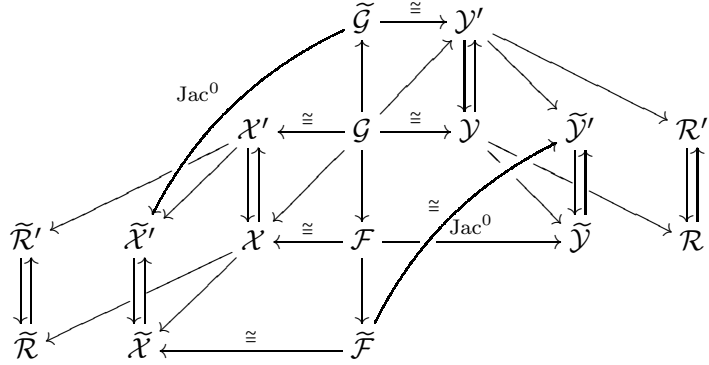
\begin{table}
\begin{tabular}{c|c}
label & defining equation \\
\hline
&\\[-0.9em]
$\mathcal{X}$ & $Y^2 Z = X^3 + \alpha(s^2, t^2) \, X^2 Z + \gamma(s^2, t^2) \,  \delta(s^2, t^2) \, X Z^2$ \\ 
$\mathcal{X}'$ &  $y^2 z = x^3 - 2 \,  \alpha(s^2, t^2) \, x^2 z + \big( \alpha(s^2, t^2)^2 - 4 \gamma(s^2, t^2) \,  \delta(s^2, t^2)\big) x z^2$ \\ 
$\mathcal{Y}$ &  $Y^2 Z = X^3 - 2 \, a(u^2, v^2) \, X^2 Z + \big( a(u^2, v^2)^2 - 4 c(u^2, v^2) \,  d(u^2, v^2)\big) X Z^2$ \\ 
$\mathcal{Y}'$ &$y^2 z = x^3 + a(u^2, v^2) \, x^2 z + c(u^2,v^2) \,  d(u^2, v^2) \, x z^2$ \\[0.2em]
\hdashline
&\\[-0.9em]
$\widetilde{\mathcal{X}}$ & $y^2 z = x^3 + S T \alpha(S, T) \, x^2 z + S^2 T^2 \gamma(S, T) \,  \delta(S,T) \, x z^2$ \\ 
$\widetilde{\mathcal{X}}'$ &  $Y^2 Z = X^3 - 2\, S T \alpha(S, T) \, X^2 Z + S^2 T^2 \big( \alpha(S, T)^2 - 4 \gamma(S, T) \,  \delta(S, T)\big) X Z^2$ \\ 
$\widetilde{\mathcal{Y}}$ &  $y^2 z = x^3 - 2 \, U V a(U, V) \, x^2 z + U^2 V^2 \big( a(U, V)^2 - 4 c(U, V) \,  d(U, V)\big) x z^2$ \\ 
$\widetilde{\mathcal{Y}}'$ &$Y^2 Z = X^3 + UV a(U, V) \, X^2 Z + U^2 V^2 c(U, V) \,  d(U, V) \, X Z^2$ \\[0.2em]
\hdashline
&\\[-0.9em]
$\widetilde{\mathcal{R}}$ & $y^2 z = x^3 + \alpha(S, T) \, x^2 z + \gamma(S, T) \,  \delta(S,T) \, x z^2$ \\ 
$\widetilde{\mathcal{R}}'$ &  $Y^2 Z = X^3 - 2 \, \alpha(S, T) \, X^2 Z +  \big( \alpha(S, T)^2 - 4 \gamma(S, T) \,  \delta(S, T)\big) X Z^2$ \\ 
$\mathcal{R}$ &  $y^2 z = x^3 - 2 \, a(U, V) \, x^2 z + \big( a(U, V)^2 - 4 c(U, V) \,  d(U, V)\big) x z^2$ \\ 
$\mathcal{R}'$ &$Y^2 Z = X^3 + a(U, V) \, X^2 Z +  c(U, V) \,  d(U, V) \, X Z^2$ \\[0.2em]
\hdashline
&\\[-0.9em]
$\mathcal{G}$ & $w^2 =  \left\lbrace \begin{array}{l}  \gamma(s^2, t^2) \, u^4 + \alpha(s^2, t^2) \, u^2v^2 + \delta(s^2, t^2) \, v^4 \\  c(u^2, v^2) \, s^4 + a(u^2, v^2) \, s^2 t^2 + d(u^2, v^2) \, t^4  \end{array}\right.$\\[0.2em]
$\mathcal{F}$ & $W^2 =  \left\lbrace \begin{array}{l} U V   \big( \gamma(s^2, t^2) \, U^2 + \alpha(s^2, t^2) \, UV + \delta(s^2, t^2) \, V^2 \big) \\ U V   \big( c(U, V) \, s^4 + a(U, V) \, s^2 t^2 + d(U, V) \, t^4 \big) \end{array}\right.$\\[0.2em]
$\widetilde{\mathcal{G}}$ & $W^2 =  \left\lbrace \begin{array}{l} S T   \big( \gamma(S, T) \, u^4 + \alpha(S, T) \, u^2v^2 + \delta(S, T) \, v^4 \big) \\ S T  \big( c(u^2, v^2) \, S^2 + a(u^2, v^2) \, S T + d(u^2, v^2) \, T^2 \big) \end{array}\right.$\\[0.2em]
$\widetilde{\mathcal{F}}$ & $w^2 =  \left\lbrace \begin{array}{l} S T U V \big( \gamma(S, T) \, U^2 + \alpha(S, T) \, UV + \delta(S, T) \, V^2 \big) \\ S T U V \big( c(U, V) \, S^2 + a(U, V) \, S T + d(U, V) \, T^2 \big) \end{array}\right.$\\[0.2em]
\hline
\end{tabular}
\caption{Defining equations for surfaces in Figure~\ref{fig:FTH-CHL_2}}
\label{tab:WEQ}
\end{table}
\par We will now determine the space of correspondences between $\widetilde{\mathcal{X}}$ and $\widetilde{\mathcal{Y}}'$. We consider the general double-quadrics given by
\beq
\label{eqn:divisor44}
 W^2 =     \big(S - c_0 T\big)  \big(c_\infty S -  T\big)\big(U - d_0 V\big)  \big(d_\infty U -  V\big)  \Big(\gamma(U, V) \, S^2 + \alpha(U, V) \, ST + \delta(U, V) \, T^2 \Big)\,,
\eeq
branched over bi-degree $(4,4)$-curves in $\mathbb{F}_0 = \mathbb{P}(S,T) \times \mathbb{P}(U,V)$. Here $\alpha, \gamma, \delta$ are homogenous polynomials of degree two and  $c_0, c_\infty, d_0, d_\infty$  are complex parameters such that $c_0 c_\infty \not = 1$ and $d_0 d_\infty \not = 1$.  The branch locus is reducible and consists of two curves of bi-degree $(1,0)$ and $(0,1)$, respectively, and a genus-one curve in the linear system $|\mathcal{O}_{\mathbb{F}_0}(2,2)|$ such that the curves intersect in ordinary rational double points.  Let $\widetilde{\mathfrak{W}}$ be the moduli space of the K3 surfaces that can be obtained as the minimal resolution of the  surfaces in Equation~(\ref{eqn:divisor44}). We note that Equation~(\ref{eqn:divisor44}) depends on 12 parameters since one has $\mathbb{P} H^0(\mathbb{F}_0, \mathcal{O}_{\mathbb{F}_0}(2,2)) = 3 \cdot 3 -1=8$  in addition to $c_0, c_\infty, d_0, d_\infty$.  Since $\mathbb{F}_0$ has a 6-dimensional group of automorphisms, the moduli space $\widetilde{\mathfrak{W}}$ is 6-dimensional. We have the following:
\begin{corollary}
The correspondences $\widetilde{\mathcal{F}}$ in Figure~\ref{fig:FTH-CHL_2} form the moduli space $\widetilde{\mathfrak{W}}$.
\end{corollary}
We now derive an explicit parametrization for the K3 surfaces in $\widetilde{\mathfrak{W}}$ using an auxiliary genus-one curve. This parametrization of the curves in the branch locus is based on a construction given by Andr\'e Weil \cite{MR717601} using the Abel-Jacobi map for genus-one curves. 
\subsection{The Abel-Jacobi map}
\label{ssec:AJM}
Let $\mathit{H}$ be a smooth curve of genus one given by $w^2 = P(x) = \sum_{i=0}^4 a_i x^i$, using the affine coordinates $(x, w) \in \mathbb{C}^2$.  For a point $(x_0,-w_0) \in \mathit{H}$ we consider the Abel-Jacobi map $J_{(x_0,-w_0)} : \mathit{H} \to \mathrm{Jac}(\mathit{H})$ which relates the algebraic curve $\mathit{H}$ to its Jacobian variety $\mathrm{Jac}(\mathit{H})$, i.e., an elliptic curve. A classical result due to Hermite states that $\mathrm{Jac}(\mathit{H}) \cong \mathit{E}$ where $\mathit{E}$ is the elliptic curve given by
\beq
\label{eqn:EC}
 \mathit{E}: \quad \eta^2 = S(\xi) = \xi^3 + f\,  \xi + g \,.
\eeq
Here, we are using the affine coordinates $(\xi,\eta) \in \mathbb{C}^2$ and Equations~(\ref{eqn:hermite0}), i.e.,
\beq
\label{eqn:hermite}
 f = - 4 a_0 a_4 + a_1 a_3 - \frac{1}{3} a_2^2 \,, \qquad g = -\frac{8}{3} a_0 a_2 a_4 + a_0 a_3^2 + a_1^2 a_4 - \frac{1}{3} a_1 a_2 a_3 + \frac{2}{27} a_2^3 \,.
\eeq
We introduce the polynomial
\beq
 R(x, x_0)  = a_4 \,x^2 x_0^2 + \frac{a_3}{2} x x_0 \big( x + x_0\big) + \frac{a_2}{6}  \big( x^2 + x_0^2\big)  + \frac{2 a_2}{3} x x_0 + \frac{a_1}{2} \big( x+x_0\big) + a_0  
\eeq
such that $R(x, x)=P(x)$. It turns out that the polynomial $P(x) P(x_0) - R(x, x_0)^2$ factors. There is a polynomial $R_1(x, x_0)$ of bi-degree $(2,2)$ such that 
\beq
\label{eqn:relat}
 \forall x, x_0: R(x, x_0)^2 +  R_1(x, x_0) \, \big(x - x_0\big)^2 - P(x) \, P(x_0) =0 \,.
\eeq
We obtain
\beq
\begin{split}
 R_1(x, x_0) & = \frac{8 a_2-3a_3^2}{12} x^2 x_0^2 + \frac{6a_1 - a_2 a_3}{6} x x_0 (x+x_0) + \frac{36 a_0-a_2^2}{36} (x^2+x_0^2)  \\
 & + \frac{36 a_0  + 9 a_1 a_3 - 5 a_2^2}{18} x x_0 +  \frac{6a_0 a_3 - a_1 a_2}{6} (x + x_0) + \frac{8 a_0 a_2-3a_1^2}{12} ,
\end{split}
\eeq 
and set $Q(x)=R_1(x, x)$. In particular, we have 
\beq
 Q(x)= \frac{1}{3} P(x) P''(x) - \frac{1}{4} P'(x)^2 \,.
\eeq
We denote the discriminants of $\mathit{H}$ and $\mathit{E}$ by $\Delta_\mathit{H} =\mathrm{Discr}_x(P)$ and $\Delta_\mathit{E} =\mathrm{Discr}_\xi(S)$, respectively, such that $\Delta_\mathit{H} =\Delta_\mathit{E}$ by construction. One also checks $\mathrm{Discr}_x(Q)=S(0)^2\mathrm{Discr}_x(P)$. From now on, we will assume that 
\beq
\label{eqn:constraint0}
 \mathrm{Discr}_x(Q) =S(0)^2\mathrm{Discr}_x(P) \not =0 \,.
\eeq 
We also set $[P, Q]= \partial_xP \cdot Q  - P \cdot \partial_xQ$. A tedious but straight-forward computation yields the following:
\begin{lemma}
\label{lem:AJM}
For a smooth curve $\mathit{H}$ of genus one given by $w^2 = \sum_{i=0}^4 a_i x^{4-i}$, the Abel-Jacobi map $J_{(x_0,-w_0)} : \mathit{H} \to \mathit{E} \cong  \mathrm{Jac}(\mathit{H})$ maps $(x, y) \mapsto (\xi,\eta)$ with
\beq
\label{eqn:AJM}
 \xi = 2 \frac{ R(x, x_0) - w w_0}{(x-x_0)^2} \,, \quad \eta =  \frac{4 w w_0 (w-w_0)}{(x-x_0)^3} - \frac{P'(x) w_0+ P'(x_0) w}{(x-x_0)^2} \quad \text{for $x\not = x_0$} \,, 
\eeq 
the point $(x_0,-w_0) \in \mathit{H}$ to the point at infinity on $\mathit{E}$, and $(x_0,w_0)$ to the point with $\xi=-Q(x_0)/P(x_0)$, $\eta= [P, Q]_{x_0}/(2w_0^3)$ if $w_0 \not = 0$. 
\end{lemma}
\par  It follows from Equation~(\ref{eqn:AJM}) that the coordinates $x$ and $\xi$ in the Abel-Jacobi map  $(\xi,\eta) = J_{(x_0,-y_0)} (x, y)$ are related by the bi-quadratic polynomial
\beq
\label{eqn:correspondence}
  \xi^2  (x - x_0)^2  - 4 \, \xi R(x, x_0)  -  4 R_1(x, x_0)   =0 \,.
\eeq 
The equation defines an algebraic correspondence between points of the two projective lines with affine coordinates $\xi$ and $x$, respectively, where -- given a point $x$ -- there are two solutions for $\xi$  in Equation~(\ref{eqn:correspondence}) and vice versa.  Equivalently, we consider Equation~(\ref{eqn:correspondence}) an affine equation of bi-degree $(2,2)$ in $\mathbb{P}^1 \times \mathbb{P}^1$ with the affine variables $(x, x_0)$ (which has genus one).  A direct computation yields the following:
\begin{lemma}
\label{lem:j}
The j-invariants of the following genus-one curves are identical:
\beq
 \eta^2 = S(\xi) \,, \qquad w^2 = P(x) \,, \qquad  \xi^2  (x - x_0)^2  - 4 \, \xi R(x, x_0)  -  4 R_1(x, x_0) =0 \,.
\eeq
\end{lemma}
In particular, Lemma~\ref{lem:j} shows that the $j$-invariant of the curve in Equation~(\ref{eqn:correspondence}) is independent of the variable $\xi$.  We have the following:
\begin{proposition}
\label{prop:RR1}
Equation~(\ref{eqn:correspondence}) is an embedding $\imath_\xi: \mathit{H} \hookrightarrow \mathit{H}' \subset \mathbb{P}^1 \times \mathbb{P}^1$ of a genus-one curve $\mathit{H}$ as a curve of bi-degree $(2,2)$ in $\mathbb{P}^1 \times \mathbb{P}^1$, given by
\beq
\label{eqn:correspondence_prop}
   \mathit{H}': \quad  \xi^2  (x - x_0)^2  - 4 \, \xi R(x, x_0)  -  4 R_1(x, x_0)   =0 \,.
\eeq 
In particular, it is a symmetric affine equation in the variables $(x, x_0)$. 
\end{proposition}
Based on results in \cites{MR61857,MR717601,MR990136} we have the following:
\begin{theorem}
\label{thm:genus-one-curves}
Let
\beq
\label{eqn:divisor22_later}
  0 = \gamma(S, T) \, U^2 + \alpha(S, T) \, UV + \delta(S, T) \, V^2 
\eeq
be a curve of genus one in the linear system $|\mathcal{O}_{\mathbb{F}_0}(2,2)|$ over $\mathbb{F}_0 \ni ([S:T],[U,V])$ with the same j-invariant as $\mathit{H}: w^2 = P(x)$. Then, there is a point in the Jacobian $\xi \in \mathrm{Jac}(\mathit{H})$ such that Equation~(\ref{eqn:divisor22_later}) is isomorphic to $\imath_\xi(\mathit{H})$.
\end{theorem}
\begin{proof}
Equation~(\ref{eqn:divisor22_later}) depends on eight parameters since one has
\beq
\mathbb{P} H^0(\mathbb{F}_0, \mathcal{O}_{\mathbb{F}_0}(2,2)) = 3 \cdot 3 -1=8 \,.
\eeq
Since $\mathbb{F}_0$ has a 6-dimensional group of automorphisms, we have two moduli. The first one is the $j$-invariant defining an elliptic curve $\mathit{E}$. It follows that  $\mathit{E} \cong \mathrm{Jac}(\mathit{H})$.  On the hand, the points in $\mathrm{Jac}(\mathit{H})$ parameterize the automorphisms of $\mathit{H}$. Thus, there is a point $\xi \in \mathrm{Jac}(\mathit{H})$ such that Equation~(\ref{eqn:divisor22_later}) is isomorphic to $\imath_\xi(\mathit{H})$.
\end{proof}
\subsection{A parametrization of K3 surfaces}
We derive an explicit parametrization for the K3 surfaces in $\widetilde{\mathfrak{W}}$. The central idea is that genus-one curves in the branch locus of K3 surfaces $\widetilde{\mathcal{F}} \in \widetilde{\mathfrak{W}}$ can be represented by equations of bi-degree $(2,2)$ that are symmetric under the interchange of $(U,V)$ and $(S,T)$ in $\mathbb{F}_0 = \mathbb{P}(S,T) \times  \mathbb{P}(U,V)$.
\par We start by using transformations in $\mathrm{PGL}(2,\mathbb{C})$ to turn the first terms of Equation~(\ref{eqn:divisor44}) into the monomial $STUV$. We then use the transformation
\beq
 \big(U,V\big) \ \mapsto \ \big(\mu U, \lambda V\big) \,, \qquad \big(S,T\big) \ \mapsto \ \big( \lambda (T - \lambda^{-1} c_\infty S), \, S- \lambda c_0 T \big),
\eeq
with parameters $\lambda, \mu \in \mathbb{C}^\times$ and $c_0, c_\infty$ with $c_0 c_\infty \not =1$, to transform an equation of the form
\beq
 W^2 = STUV  \Big(\gamma'(U, V) \, S^2 + \alpha'(U, V) \, ST + \delta'(U, V) \, T^2 \Big)
\eeq
for polynomials $\alpha', \gamma', \delta'$ with coefficients $\alpha'_2, \dots, \gamma'_0$ into the equation
\beq
\label{eqn:K3prelim}
\begin{split}
  W^2 =   &\,  \big(S - \lambda c_0 T\big)  \big(T - \lambda^{-1} c_\infty S \big) \, U V \,  \Big(  \big( \gamma_2 S^2 + \lambda \alpha_2 ST + \lambda^2 \gamma_0 T^2 \big) U^2 \\
  + &   \big( \lambda \alpha_2 S^2 + \lambda^2 \alpha_1 ST + \lambda^3 \alpha_0 T^2 \big) UV +   \big( \lambda^2 \gamma_0 S^2 + \lambda^3 \alpha_0 ST + \lambda^4 \delta_0 T^2 \big) V^2 \Big) \,, 
  \end{split} 
\eeq 
where $\mu, c_0, c_\infty$ have been chosen such that the new coefficients $\alpha_2, \dots, \gamma_0$ do not depend on $\lambda$ and satisfy $\gamma_1= \alpha_2$,  $\delta_2=\gamma_0$, $\delta_1=\alpha_0$. Thus, we have proved the following:
\begin{proposition}
\label{prop:N0}
The moduli space $\widetilde{\mathfrak{W}}$ is given by the K3 surfaces obtained from the minimal resolution of the double-quadrics
\beq
\label{eqn:K3_in_Mpprime}
  W^2 =  \big(S - c_0 T\big)  \big(T - c_\infty S \big) \, U V \,\big( \gamma(U, V) \, S^2 + \alpha(U, V) \, ST + \delta(U, V) \, T^2 \big) \,.
\eeq  
Here, $c_0, c_\infty \in \mathbb{C}$ with $c_0 c_\infty \not = 1$ and $\alpha, \gamma, \delta$ are given by
\beq
\alpha  =  \alpha_2 x^2 + \alpha_1 xy + \alpha_0 y^2,  \quad
\gamma  = \gamma_2 x^2 + \alpha_2 xy + \gamma_0 y^2, \quad
\delta  = \gamma_0 x^2 + \alpha_0 xy + \delta_0 y^2 \,.
\eeq
\end{proposition}
We make the following:
\begin{remark}
Proposition~\ref{prop:N0} implies that we can assume 
\beq
\label{eqn:divisor22}
  \gamma(S, T) \, U^2 + \alpha(S, T) \, UV + \delta(S, T) \, V^2  = \gamma(U, V) \, S^2 + \alpha(U, V) \, ST + \delta(U, V) \, T^2 
\eeq
for all $S, T, U, V$.
\end{remark}
By a slight abuse of notation we will also denote by $\widetilde{\mathcal{F}} ( \delta_0,  \alpha_0,  \gamma_0,  \alpha_1,  \alpha_2,  \gamma_2;  c_0, c_\infty)$ the smooth complex surface obtained as the minimal resolution of  Equation~(\ref{eqn:K3_in_Mpprime}). We have the following symmetries:
\begin{lemma}
\label{symmetries1}
One has the following isomorphisms of K3 surfaces:
\begin{itemize}
\item [(a)] $\scalemath{0.9}{\widetilde{\mathcal{F}} ( \delta_0,  \alpha_0,  \gamma_0,  \alpha_1,  \alpha_2,  \gamma_2;  c_0, c_\infty) \ \simeq \
\widetilde{\mathcal{F}} ( \lambda \delta_0,  \lambda\alpha_0,  \lambda\gamma_0,  \lambda\alpha_1, \lambda \alpha_2,  \lambda\gamma_2;  c_0, c_\infty )}$,
\item [(b)] $\scalemath{0.9}{\widetilde{\mathcal{F}}  ( \delta_0,  \alpha_0,  \gamma_0,  \alpha_1,  \alpha_2,  \gamma_2;  c_0, c_\infty) \ \simeq \
\widetilde{\mathcal{F}} ( \mu^4\delta_0,  \mu^3\alpha_0,  \mu^2\gamma_0,  \mu^2 \alpha_1, \mu \alpha_2,  \gamma_2 ;   \mu c_0, \mu^{-1}c_\infty )}$,
\item [(c)] $\scalemath{0.9}{\widetilde{\mathcal{F}}  ( \delta_0,  \alpha_0,  \gamma_0,  \alpha_1,  \alpha_2,  \gamma_2;  c_0, c_\infty) \ \simeq \
\widetilde{\mathcal{F}} ( \gamma_2,  \alpha_2,  \gamma_0,   \alpha_1, \alpha_0,  \delta_0 ;   c_0^{-1}, c_\infty^{-1} )}$,
\item [(d)] $\scalemath{0.9}{\widetilde{\mathcal{F}}  ( \delta_0,  \alpha_0,  \gamma_0,  \alpha_1,  \alpha_2,  \gamma_2;  c_0, c_\infty) \ \simeq \
\widetilde{\mathcal{F}} (\delta'_0,  \alpha'_0,  \gamma'_0,  \alpha'_1,  \alpha'_2,  \gamma'_2;   -c_0, -c_\infty )}$,
\end{itemize}
for $\lambda, \mu\in \mathbb{C}^\times$ and
\beq
\begin{split}
 \delta'_0 & = \gamma_2 c_0^4 + 2 \alpha_2 c_0^3 +(\alpha_1 + 2 \gamma_0) c_0^2 + 2 \alpha_0 c_0 + \delta_0,\\
 \alpha'_0 & = (\alpha_2 c_\infty + 2 \gamma_2) c_0^3 + ((\alpha_1 + 2 \gamma_0) c_\infty + 3 \alpha_2) c_0^2 + (3 \alpha_0 c_\infty + \alpha_1 + 2 \gamma_0) c_0 
 + 2 \delta_0 c_\infty + \alpha_0,\\
 \gamma'_0 & = (\gamma_0 c_\infty^2 + \alpha_2 c_\infty + \gamma_2) c_0^2 + (\alpha_0 c_\infty^2 + \alpha_1 c_\infty+ \alpha_2) c_0 + \delta_0 c_\infty^2 + \alpha_0 c_\infty + \gamma_0,\\
 \alpha'_1 & = (\alpha_1 c_\infty^2 + 4 \alpha_2 c_\infty+ 4 \gamma_2) c_0^2 + (4 \alpha_0 c_\infty^2 + 2(\alpha_1 + 4 \gamma_0) c_\infty+ 4 \alpha_2) c_0\\
 & + 4 \delta_0 c_\infty^2 + 4 \alpha_0 c_\infty + \alpha_1 ,\\
 \alpha'_2 & = (\alpha_0 c_0 + 2 \delta_0) c_\infty^3 + ((\alpha_1 + 2 \gamma_0) c_0 + 3 \alpha_0) c_\infty^2 + (3 \alpha_2 c_0 + \alpha_1 + 2 \gamma_0) c_\infty 
 + 2 \gamma_2 c_0 + \alpha_2,\\
 \gamma'_2 & = \delta_0 c_\infty^4 + 2 \alpha_0 c_\infty^3 +(\alpha_1 + 2 \gamma_0) c_\infty^2 + 2 \alpha_2 c_\infty + \gamma_2,\\
\end{split}
\eeq
\end{lemma}
\begin{proof}
Part (1) follows from rescaling $W \mapsto W/\sqrt{\lambda}$, similarly part (2) follows from the rescaling $(T,V) \mapsto (\mu T, \mu V)$. For the remaining parts, the proof follows by constructing a suitable automorphism of $\mathbb{F}_0$ in Equation~(\ref{eqn:K3_in_Mpprime}) that gives to the change of parameters. For part (3) this is given by interchanging $S \leftrightarrow T$  and $U \leftrightarrow V$. For part (4) the transformation is given by
\beq
 U  = S ' + c_0 T', \quad V = c_\infty S' + T' , \quad S  = U ' + c_0 V', \quad T = c_\infty U' + V' .
\eeq
\end{proof}
\par We now give an explicit parametrization for the K3 surfaces in $\widetilde{\mathfrak{W}}$ using the Abel-Jacobi map from Section~\ref{ssec:AJM}. We have the following:
\begin{theorem}
\label{thm:subspace1}
The moduli space $\widetilde{\mathfrak{W}}$ of correspondences $\widetilde{\mathcal{F}}$ in Figure~\ref{fig:FTH-CHL_2} is given by the K3 surfaces obtained as the minimal resolution of the double-quadrics
\beq
\label{eqn:K3_in_subspace1}
  W^2 =  \big(S - c_0 T\big)  \big(T - c_\infty S \big) \, U V \,\big( \gamma(U, V) \, S^2 + \alpha(U, V) \, ST + \delta(U, V) \, T^2 \big) \,.
\eeq  
Here, $c_0, c_\infty \in \mathbb{C}$ with $c_0 c_\infty \not = 1$, the polynomials $\alpha, \gamma, \delta$ are given by
\beq
\forall x, x_0: \ \gamma(x, 1) \, x_0^2 + \alpha(x, 1) \, x_0 + \delta(x, 1)  = \xi^2  (x - x_0)^2  - 4 \, \xi R(x, x_0)  -  4 R_1(x, x_0) ,
\eeq
and the polynomials $R(x, x_0)$ and $R_1(x, x_0)$ are 
\beq
\begin{split}
 R(x, x_0)  & = x^2 x_0^2 + \frac{a_2}{6}  \big( x^2 + x_0^2\big)  + \frac{2 a_2}{3} x x_0 + \frac{a_1}{2} \big( x+x_0\big) + a_0 \,, \\
 R_1(x, x_0) & = \frac{2 a_2}{3} x^2 x_0^2 + a_1 x x_0 (x+x_0) + \frac{36 a_0-a_2^2}{36} (x^2+x_0^2)  \\
 & + \frac{36 a_0  - 5 a_2^2}{18} x x_0 -  \frac{a_1 a_2}{6} (x + x_0) + \frac{8 a_0 a_2-3a_1^2}{12} \, ,
\end{split}
\eeq 
for the smooth genus-1 curve $\mathit{H}: w^2 = x^4 + \sum_{i=0}^2 a_i x^{i}$ and $\xi \in \mathrm{Jac}(\mathit{H})$.
\end{theorem}
\begin{remark}
The parameterization in Theorem~\ref{thm:subspace1} is given by the six parameters $c_0, c_\infty, a_0, a_1, a_2, \xi$ such that  $c_0 c_\infty \not = 1$, the curve $\mathit{H}: w^2 = x^4 + \sum_{i=0}^2 a_i x^{i}$ is smooth, and $\xi \in \mathrm{Jac}(\mathit{H})$. The curve $\mathit{H}$ is smooth if its discriminant does not vanish, i.e.,
\beqn
 16 a_0 a_2^4 - 4 a_1^2 a_2^3 - 128 a_0^2 a_2^2 + 144 a_0 a_1^2 a_2 - 27 a_1^4 + 256 a_0^3 \neq 0 .
\eeqn
The elliptic curve $\mathit{E} = \mathrm{Jac}(\mathit{H})$  in $\mathbb{P}^2=\mathbb{P}(Z_1,  Z_2, Z_3)$ is given by
\beq
\label{eqn:EC2}
 \mathit{E}: \quad  Z_2^2 Z_3= Z_1^3 - \left(4 a_0   + \frac{1}{3} a_2^2  \right)  Z_1^2 Z_3 + \left( a_1^2 + \frac{2}{27} a_2^3  -\frac{8}{3} a_0 a_2 \right) Z_3^3.
\eeq
\end{remark}
\begin{proof}
We apply Proposition~\ref{prop:N0}, Theorem~\ref{thm:genus-one-curves}, and Proposition~\ref{prop:RR1}. One then checks that in terms of the curve $\mathit{H}$ given by
$w^2 = P(x) = \sum_{i=0}^4 a_i x^i$ and the Jacobian $\mathrm{Jac}(\mathit{H})$ given by Equation~(\ref{eqn:EC}) one has
\beq
 2 \alpha_0 \gamma_2 - \alpha_1 \alpha_2 + 2 \alpha_2 \gamma_0 = - 8 a_3 \eta^2 .
\eeq
Using Lemma~\ref{symmetries1} (4) one checks that the equation
\beq
 2 \alpha'_0 \gamma'_2 - \alpha'_1 \alpha'_2 + 2 \alpha'_2 \gamma'_0 = 0 .
\eeq
is linear in $c_0$ and can be solved such that $c_0$ is a rational function in $c_\infty$ where both the numerator and denominator have degree 3 in $c_\infty$ with coefficients quadratic in $\mathbb{Z}[ \delta_0,  \alpha_0,  \gamma_0,  \alpha_1,  \alpha_2,  \gamma_2]$. Thus, for a smooth curve $\mathit{H}$ of genus one we can assume $a_4 \not=0$ and then shift the coordinate $x$ to obtain $a_3=0$. Moreover, Lemma~\ref{symmetries1} (1) provides an overall scaling that can be used to have $a_4=1$.
\end{proof}
\section{A second subspace of dimension six}
\label{sec:specialization2}
The Jacobian elliptic K3 surfaces $\mathcal{X}$ with $\mathrm{MW}(\mathcal{X}, \pi_\mathcal{X}) \cong \mathbb{Z}/2\mathbb{Z}$ in Equation~(\ref{eqn:X}) represent F-theory backgrounds with discrete flux, or, equivalently, elements of the moduli space $\mathfrak{M}_{H\oplus N}$. Based on Corollary~\ref{cor:6dim_vGS}, imposing the existence of a second (commuting) van~Geemen-Sarti involution on $\mathcal{X}$ implies that the Jacobian elliptic fibration has 12 singular fibers of type $I_2$, a Mordell-Weil group $(\mathbb{Z}/2\mathbb{Z})^2$, and a lattice polarization that extends to 
\beq
 H\oplus D_4(-1)^{\oplus 2} \oplus  A_1(-1)^{\oplus 4} \cong H\oplus D_6(-1) \oplus  A_1(-1)^{\oplus 6} \cong  \langle 2 \rangle \oplus \langle -2 \rangle\oplus D_4(-1)^{\oplus 3} .
\eeq
The equivalence of the lattices was proven in \cite{Clingher:2021}. The existence of commuting van~Geemen-Sarti involutions means that there are two independent 2-torsion sections in $\mathrm{MW}(\mathcal{X}, \pi_\mathcal{X})$ and implies that the polynomial $B(s, t)$ in Equation~(\ref{eqn:X}) satisfies $B = A^2-C^2$ for some homogeneous polynomial $C(s, t)$ of degree 4; see proof of Corollary~\ref{cor:6dim_vGS}. In turn, the factorization $4B = (A+C)(A-C)$ determines a marking of the eight fibers of type $I_2$ over $B=0$ on $\mathcal{X}$ with the double cover $\mathcal{G}$. Observing that the coefficient of $x z^2$ in Equation~(\ref{eqn:Xprime}) then becomes a perfect square, we obtain a canonical marking of $\mathcal{X}'$ with another double cover $\mathcal{G}'$. The double-quadrics $\mathcal{G}'$ has a particular simple form which we will determine presently. Changing the ruling and computing the relative Jacobian fibration we obtain the K3  surface $\mathcal{Y}'$ admitting two commuting antisymplectic involutions due to the symmetry of $\mathcal{G}'$. We already constructed the K3 surfaces $\mathcal{Y}'$ and their quotients $\mathcal{Y}, \widetilde{\mathcal{Y}}$ in Propositions~\ref{prop7} and~\ref{prop5}. Thus, we have the following:
\begin{corollary}
\label{cor:subspace2}
The moduli space of F-theory models with discrete flux admitting commuting van Geemen-Sarti involutions is dual to the moduli space of the CHL string admitting 2 antisymplectic involutions induced by involutions on the base curve. 
\end{corollary}
\par Figure~\ref{fig:FTH-CHL} can then be extended. We obtain Figure~\ref{fig:FTH-CHL_3} where now the rational elliptic surface $\mathcal{R}$ -- due to the extended symmetry of $\mathcal{Y}$ -- must satisfy
\beq
\label{eqn:MW_special_2}
 \mathrm{MW}(\mathcal{R}, \pi_\mathcal{R}) \ \cong \ D_4^\vee .
\eeq
We note that Equation~(\ref{eqn:MW_special_2})  is precisely the first case considered earlier in~(\ref{eqn:casesRES}). We have the following:
\begin{proposition}
\label{prop:situation2}
In the situation of Corollary~\ref{cor:subspace2} there exist the algebraic correspondences in  Figure~\ref{fig:FTH-CHL_3}. The  defining equations for the surfaces in Figure~\ref{fig:FTH-CHL_3}  are given by Table~\ref{tab:WEQ}. Here,  $A, C$ are general homogeneous polynomials of degree four, and $a_0, \dots, a_4$ of degree 1 are defined by requiring
\beq
\label{eqn:hermite00_3}
  A(s, t) \, \frac{U-V}{2} - C(s, t) \, \frac{U+V}{2} =  \  a_0(U,V) \, s^4 + a_1(U,V) \, s^3 t + \dots + a_4(U,V) \, t^4 
\eeq
for all $s, t, U, V$.  The polynomials $f, g$ are given by
\beq
\label{eqn:hermite0b}
 f = - 4 a_0 a_4 + a_1 a_3 - \frac{1}{3} a_2^2 \,, \qquad g = -\frac{8}{3} a_0 a_2 a_4 + a_0 a_3^2 + a_1^2 a_4 - \frac{1}{3} a_1 a_2 a_3 + \frac{2}{27} a_2^3 \,.
\eeq
In particular, we have:
\begin{enumerate}
\item $\mathcal{X}, \mathcal{F} \in \mathfrak{M}_{\langle 2 \rangle \oplus \langle -2 \rangle\oplus D_4(-1)^{\oplus 3}}$, $\mathcal{X}', \mathcal{G}, \mathcal{Y} \in \mathfrak{M}_{H \oplus K_0(-1)}$, $\widetilde{\mathcal{Y}} \in \mathfrak{M}_{H \oplus D_4(-1)^{\oplus 3}}$,
\item $\mathcal{G}', \mathcal{G}, \mathcal{F}$ are double-quadrics  over $\mathbb{F}_0$, 
\item $\mathcal{R}$ is a rational elliptic surface with $\mathrm{MW}(\mathcal{R}, \pi_\mathcal{R})  \cong \ D_4^\vee$.
\end{enumerate}
\end{proposition}
\begin{proof}
One uses the explicit constructions for double covers in Section~\ref{ssec:double_covers} and Section~\ref{sec:VGS_K3} to construct all surfaces in Figure~\ref{fig:FTH-CHL_3} explicitly.  The Gram matrix of $K_0$ was given in Equation~(\ref{eqn:Gram_matrix_1}). The existence of sections that imply isomorphisms $\mathcal{G}' \cong \mathcal{Y}'$ and $\mathcal{G} \cong \mathcal{Y}$ is proved as in Lemma~\ref{lem:sections}. The existence of sections that implies $\mathcal{X}' \cong \mathcal{G}$ follows immediately from the equation for $\mathcal{G}$ in Table~\ref{tab:WEQ_2}.
\end{proof}
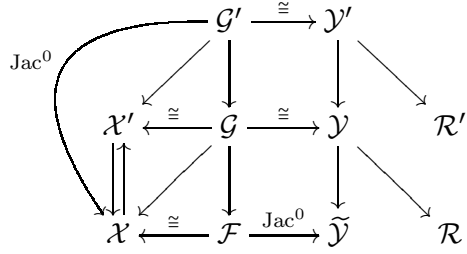
\begin{figure}
\beqn
\xymatrix{
&   \mathcal{G}'   \ar@/_4pc/[ddl]_{\mathrm{Jac}^0}  \ar[dl]  \ar[d]  \ar[r]^{\cong} & \mathcal{Y}'  \ar[d] \ar[dr]\\
 \mathcal{X}'  \ar@<-2pt>[d]  &  \mathcal{G} \ar[l]_{\cong} \ar[dl]  \ar[d]  \ar[r]^{\cong} & \mathcal{Y}  \ar[d] \ar[dr] & \mathcal{R}'\\
 \mathcal{X}  \ar@<-2pt>[u]  &  \mathcal{F} \ar[l]_{\cong}   \ar[r]^{\mathrm{Jac}^0} & \widetilde{\mathcal{Y}} &  \mathcal{R}\\
 }
\eeqn
\caption{Extension of the F-theory/CHL string duality (Case II)}
\label{fig:FTH-CHL_3}
\end{figure}
 \begin{table}
\begin{tabular}{c|c}
label & defining equation \\
\hline
&\\[-0.9em]
$\mathcal{X}$ & $Y^2 Z = X^3 - A(s, t) \, X^2 Z + \big(A(s, t)^2-C(s, t)^2\big) \, X Z^2/4$ \\ 
$\mathcal{X}'$ &  $y^2 z = x^3 + 2 A(s, t) \, x^2 z + C(s, t)^2 x z^2$ \\ 
\hdashline
&\\[-0.9em]
$\mathcal{G}'$ & $W^2 =  \left\lbrace \begin{array}{l} C(s, t) \, \tilde{u}^4 +  2 A(s, t)\, \tilde{u}^2 \tilde{v}^2  + C(s, t) \, \tilde{v}^4
\\ a_4\big( (\tilde{u}^2 - \tilde{v}^2)^2,   (\tilde{u}^2 + \tilde{v}^2)^2\big) \, s^4 +\dots +  a_0\big( (\tilde{u}^2 - \tilde{v}^2)^2,   (\tilde{u}^2 + \tilde{v}^2)^2\big)  \, t^4  \end{array} \right.$\\[0.2em]
$\mathcal{G}$ & $w^2 =  \left\lbrace \begin{array}{l}  (u^2-v^2)  \big( \frac{A(s, t)-C(s, t)}{2} \, u^2- \frac{A(s, t) + C(s, t)}{2} \, v^2 \big)  \\ 
  (u^2-v^2)  \big(  a_4(u^2, v^2) \, s^4 +\dots + a_0(u^2, v^2) \, t^4 \big) \end{array}\right.$\\[0.2em]
$\mathcal{F}$ & $W^2 =  \left\lbrace \begin{array}{l} U V  (U-V)  \big( \frac{A(s, t)-C(s, t)}{2} \, U - \frac{A(s, t) + C(s, t)}{2} \, V \big) \\
 U V  (U-V) \big( a_4(U, V) \, s^4 +\dots + a_0(U, V) \, t^4\big) \end{array}\right.$\\[0.2em]
\hdashline
&\\[-0.9em]
$\mathcal{Y}'$ & $Y^2 Z = X^3 + f\big( (\tilde{u}^2 - \tilde{v}^2)^2,   (\tilde{u}^2 + \tilde{v}^2)^2 \big) \, X Z^2 + g\big( (\tilde{u}^2 - \tilde{v}^2)^2,   (\tilde{u}^2 + \tilde{v}^2)^2 \big)  \, Z^3$ \\ 
$\mathcal{Y}$ & $y^2 z = x^3 + (u^2-v^2)^2 f(u^2, v^2) \, x z^2 + (u^2-v^2)^3 g(u^2, v^2) \, z^3$ \\ [0.1em]
$\widetilde{\mathcal{Y}}$ & $Y^2 Z = X^3 + U^2 V^2 (U-V)^2 f(U, V) \, X Z^2 + U^3 V^3 (U-V)^3 g(U, V) \, Z^3$ \\ 
\hdashline
&\\[-0.9em]
$\mathcal{R}'$ & $y^2 z = x^3 +  f(u^2, v^2) \, x z^2 + g(u^2, v^2) \, z^3$ \\ [0.1em]
$\mathcal{R}$ & $Y^2 Z = X^3 + (U-V)^2 f(U, V) \, X Z^2 + (U-V)^3 g(U, V) \, Z^3$ \\ 
\hline
\end{tabular}
\caption{Defining equations for surfaces in Figure~\ref{fig:FTH-CHL_3}}
\label{tab:WEQ_2}
\end{table}
\subsection{The double \texorpdfstring{$4\mathcal{H}$}{4H}-surfaces}
We will now determine the space of correspondences $\mathcal{F}$ for the K3 surfaces $\mathcal{X}$ and $\widetilde{\mathcal{Y}}$ in Figure~\ref{fig:FTH-CHL_3}.
\par Let $\mathcal{F}$ be a double cover of the Hirzebruch surface $\mathbb{P}^1 \times \mathbb{P}^1$ branched over the union of four bi-degree $(1,1)$ curves satisfying a certain generality condition.  Such a surface $\mathcal{F}$ is also called \emph{a double $4\mathcal{H}$-surface}. We construct a geometric model as follows: using the coordinates $U$ and $s$ on $\mathbb{P}^1_{(U)} \times \mathbb{P}^1_{(s)}$ a double cover is given by 
\beq
\label{eqn:definingF}
 y^2 = \prod_{k=1}^4 \Big(  \rho_1^{(k)} s U +  \rho_2^{(k)} U +  \rho_3^{(k)} s +  \rho_4^{(k)}  \Big) \,,
\eeq 
for complex  parameters $\rho^{(i)}_j$ with $i, j \in \{1,2,3,4\}$. We denote by $H_1, \dots, H_4$ the four  different rational curves of bi-degree $(1, 1)$ and impose the following \emph{genericity} conditions: (1) every $H_i$ is irreducible, (2) $H_i \cap H_j$ consists of two different points for all $i \neq j$, and (3) for any three different indices $i, j, k$ we have $H_i \cap H_j \cap H_k = \emptyset$. This is precisely the family considered in \cite{MR1871336}, and it was proven there that under the above conditions $\mathcal{F}$ is a K3 surface.  If we define the quadratic polynomials
\beq
\begin{split}
 P^{(i, j)} & = \Big(\rho^{(i)}_1 \rho^{(j)}_3 - \rho^{(i)}_3 \rho^{(j)}_1\Big) s^2  +  \Big(\rho^{(i)}_2 \rho^{(j)}_4 - \rho^{(i)}_4 \rho^{(j)}_2\Big)  \\
 & + \Big( \rho^{(i)}_1 \rho^{(j)}_4 + \rho^{(i)}_2 \rho^{(j)}_3 -\rho^{(i)}_3 \rho^{(j)}_2 -\rho^{(i)}_4 \rho^{(j)}_1   \Big) s \,,
\end{split} 
\eeq 
then Equation~(\ref{eqn:definingF}) can be brought into the Weierstrass form
\beq
\label{eqn:definingF_WE}
 Y^2 Z = X  \Big( X - P^{(1,2)} P^{(3,4)} Z \Big)   \Big(  X - P^{(1,3)}  P^{(2,4)}  Z \Big)  \,,
\eeq
which coincides with the equation for $\mathcal{X}$ in Table~\ref{tab:WEQ_2} in the affine chart $t=1$ with
\beqn
  P^{(1,2)} P^{(3,4)} = \frac{A+C}{2}, \qquad  P^{(1,3)}  P^{(2,4)} = \frac{A -C}{2}.
\eeqn
The following is then immediate:
\begin{lemma}
\label{lem:JacF}
The generic K3 surface $\mathcal{F}$  admits a Jacobian elliptic fibration with 12 singular fibers of type $I_2$ and a Mordell-Weil group of sections $(\mathbb{Z}/2\mathbb{Z})^2$. 
\end{lemma}
\begin{proof}
Given the Weierstrass model in Equation~(\ref{eqn:definingF_WE}) the statement is checked by explicit computation.
\end{proof}
\par Given any two distinct complex parameters $\mu, \nu \in \mathbb{C}$ with $\mu \not = \nu$, Equation~(\ref{eqn:definingF_WE}) can be brought into the form
\beqn
  y^2   = \Big(\xi + \mu \Big) \Big(\xi + \nu \Big) 
  \Big(\big( P^{(1,2)} P^{(3,4)} - P^{(1,3)} P^{(2,4)} \big) \xi + \big( \mu P^{(1,2)} P^{(3,4)} - \nu P^{(1,3)} P^{(2,4)} \big)  \Big) ,
\eeqn
which in turn can be re-written as
\beq
\label{eqn:g1fibration}
  y^2   = \sum_{i=0}^4  (\xi+\mu)(\xi+\nu) \, A_i(\xi,\mu,\nu) \, u^i \,,
\eeq
with $A_i(\xi,\mu,\nu) = a_{i,1} \xi +  a_{i,2} \mu + a_{i,3} \nu$ for $0 \le i \le 4$. The coefficients are of the form
\beq
 a_{i, j} = \sum_{k_1, \dots, k_4} \alpha_{i, j}(\vec{k}) \;  \rho^{(1)}_{k_1} \rho^{(2)}_{k_2} \rho^{(3)}_{k_3} \rho^{(4)}_{k_4} 
\eeq 
with $\alpha_{i, j}(\vec{k})  \in \{ 0, \pm1 \}$. Considering $\xi$ the affine coordinate of a projective line and $(u, y)$ the affine coordinates of a genus-one fiber in $\mathbb{P}^2$, it follows that Equation~(\ref{eqn:g1fibration}) induces a genus-one fibration on the K3 surface $\mathcal{F}$. Using \cite{MR3995925}*{Prop.~3.3} it follows immediately:
\begin{lemma}
\label{prop:KSTTfibration}
The very general K3 surface $\mathcal{F}$  admits a genus-one fibration (without section and) with three singular fibers of type $I_0^*$ and six singular fibers of type $I_1$.
\end{lemma}
The existence of such a genus-one fibration with three singular fibers of type $I_0^*$ on the K3 surface $\mathcal{F}$ allowed the authors in \cite{MR1871336}*{Thm.~1} to compute the lattice polarization of the family. We have:
\begin{proposition}
\label{prop:trans_F}
For generic parameters $\rho^{(i)}_j$ with $i, j \in \{1,2,3,4\}$ the K3 surface $\mathcal{F}$ is endowed with a canonical polarization given by the rank-fourteen lattice
\beq
  \langle 2 \rangle \oplus  \langle - 2 \rangle \oplus D_4(-1)^{\oplus 3} \,.
\eeq
Moreover, the transcendental lattice is $H(2)^{\oplus 2} \oplus \langle -2 \rangle^{\oplus 4}$ of signature $(2,6)$.
\end{proposition}
We have the immediate:
\begin{corollary}
The double $4\mathcal{H}$-surfaces $\mathcal{F}$ form the moduli space of  correspondences between the K3 surfaces $\mathcal{X}$ and $\widetilde{\mathcal{Y}}$ in Figure~\ref{fig:FTH-CHL_3}.
\end{corollary}
\subsection{Double covers of a cubic and three lines}
In this section we describe the geometry of the K3 surfaces $\widetilde{\mathcal{Y}}$ in Figure~\ref{fig:FTH-CHL_3}. They represent the CHL string backgrounds dual to the  F-theory with additional symplectic involution.
\par Let $\bar{\mathcal{Y}}$ be the double cover of the projective plane $\mathbb{P}^2 =\mathbb{P}(Z_1, Z_2, Z_3)$ branched over the union of three lines $\ell_1, \ell_2, \ell_3$ coincident in a point and a cubic $\mathit{E}$. We call such a configuration \emph{generic} if the cubic is smooth and meets the three lines in nine distinct points. In particular, the cubic does not meet the point of coincidence of the three lines. We construct a geometric model as follows: we use a suitable projective transformation to move the line $\ell_3$ to $\ell_3 = \mathrm{V}(Z_3)$. We then mark three distinct points $q_0$, $q_1$, and $q_\infty$ on $\ell_3$ and use a M\"obius transformation to move these points to $[Z_1 : Z_2 : Z_3] = [ 0 : 1 : 0]$, $[1:1:0]$, and $[1: 0: 0]$. Let $q_1: [1:1:0]$ be the point of coincidence of the three lines. Up to scaling, the three lines, coincident in $q_1$, are then given by
\beq
  \ell_1 = \mathrm{V}\big( Z_1 - Z_2  + \mu Z_3 \big) \,, \qquad  \ell_2 = \mathrm{V}\big( Z_1 - Z_2  + \nu Z_3 \big)  \,, \qquad  \ell_3 = \mathrm{V}\big(Z_3\big) \,,
\eeq  
for some parameters $\mu, \nu$ with $\mu \neq \nu$. Let the cubic $\mathit{E}=\mathrm{V} ( C(Z_1, Z_2, Z_3) )$ intersect the line $\ell_3$ at $q_0$, $q_\infty$, and at the point $[-d_2:c_1:0] \neq [1:1:0]$. Thus, we have 
\beq
\label{eqn:cubic}
 C = e_3 Z_3^3 + \Big( d_0 Z_1 + e_1 Z_2 \Big) Z_3^2 +  \Big( c_0 Z_1^2 + d_1 Z_1 Z_2 + e_2 Z_2^2 \Big) Z_3 + Z_1 Z_2 \Big(c_1 Z_1 + d_2 Z_2\Big)\,.
\eeq 
This can be written as
\beq
\label{eqn:cubic2}
 C=\Big( c_1 Z_2 + c_0 Z_3\Big) Z_1^2 + \Big( d_2 Z_2^2 +d_1 Z_2 Z_3 + d_0 Z_3^2\Big) Z_1 + \Big(e_2 Z_2^2 + e_1 Z_2 Z_3 + e_0 Z_3^2\Big) Z_3,
\eeq  
such that in $\mathbb{WP}_{(1,1,1,3)} =\mathbb{P}(Z_1, Z_2, Z_3, Y)$ the surface $\bar{\mathcal{Y}}$ is given by
\beq
\label{eqn:defining_eqn_3lines+cub_1st}
 Y^2  = \big( Z_1 - Z_2  + \mu Z_3 \big) \big( Z_1 - Z_2  + \nu Z_3 \big) Z_3 \; C(Z_1, Z_2 Z_3) \,,
\eeq
for parameters $\mu, \nu, c_0, c_1, d_0, d_1, d_2, e_0, e_1, e_2$ with $c_1 \neq 0, c_1 + d_2 \neq 0$, $\mu \neq \nu$, and a smooth cubic $\mathit{E}$  that intersects each line $\ell_1, \ell_2, \ell_3$ in three distinct points. We have the following:
\begin{lemma}
\label{lem:specialization_cubic}
The cubic $\mathit{E}$ is tangent to the line $\ell_3$ at $q_0$ if and only if $d_2=0$ and the remaining parameters are generic. The cubic $\mathit{E}$ is singular at $q_0$ if and only if $d_2=e_2=0$ and the remaining parameters are generic.
\end{lemma}
In~\cite{Clingher:2020baq} the authors proved that the coefficients of the curves in the branch locus can be normalized as follows:
\begin{lemma}
\label{lem:double_cover}
Let $\bar{\mathcal{Y}}$ be the double cover of the projective plane $\mathbb{P}^2 = \mathbb{P}(Z_1, Z_2, Z_3)$ branched over three lines coincident in a point and a generic cubic. There are affine parameters $(d_2, \mu, c_0, e_2, d_0, e_1, e_0) \in \mathbb{C}^7$, unique up to the action given by
\beq
\label{eqn:rescaling}
 \Big(  d_2, \ \mu, \ c_0, \ e_2, \ d_0, \  e_1,  \, e_0 \Big) \ \mapsto \  
 \Big(  d_2, \ \Lambda \mu, \ \Lambda c_0, \ \Lambda e_2, \ \Lambda^2  d_0, \ \Lambda^2 e_1,  \, \Lambda^3 e_0 \Big) 
\eeq 
with $\Lambda \in \mathbb{C}^{\times}$, such that $\bar{\mathcal{Y}}$ in $\mathbb{WP}_{(1,1,1,3)} =\mathbb{P}(Z_1, Z_2, Z_3, Y)$ is obtained by
\beq
\label{eqn:defining_eqn_3lines+cub}
\begin{split}
 Y^2 & \,  = \big( Z_1 - Z_2  + \mu Z_3 \big) \big( Z_1 - Z_2  + \nu Z_3 \big) Z_3 \\
  & \times   \left( \Big( Z_2 + c_0 Z_3\Big) Z_1^2 + \Big( d_2 Z_2^2 + d_0 Z_3^2\Big) Z_1 + \Big(e_2 Z_2^2 + e_1 Z_2 Z_3 + e_0 Z_3^2\Big) Z_3 \right) \,,
\end{split} 
\eeq
with $\mu + \nu = (1+d_2/2)(c_0 + e_2)$ and $d_2 \neq -1$.
\end{lemma} 
We denote by $\widetilde{\mathcal{Y}}$ the surface obtained as the minimal resolution of $\bar{\mathcal{Y}}$. Since $\widetilde{\mathcal{Y}}$ is the resolution of a double-sextic surface, it is a K3 surface.  We will now construct a Jacobian elliptic fibrations on it to establish the connection with $\widetilde{\mathcal{Y}}$ in Figure~\ref{fig:FTH-CHL_3}:
\begin{lemma}
\label{lem:defining_eqn_3lines+cub}
A generic K3 surface $\widetilde{\mathcal{Y}}$  admits a Jacobian elliptic fibration with the singular fibers $3 I_0^* + 6 I_1$ and a trivial Mordell-Weil group.  
\end{lemma}
\begin{proof}
The pencil of lines $(Z_1 - Z_2) - t Z_3=0$ for $t\in \mathbb{C}$ through the point $q_1=[1:1:0]$ induces an elliptic fibration on $\widetilde{\mathcal{Y}}$.  We refer to this fibration as \emph{the standard fibration}. When substituting $Z_1 = X$, $Z_2=X-(c_1+d_2) (t+\mu)(t+\nu) t$, and $Z_3= (c_1+d_2)(t+\mu)(t+\nu)$ into Equation~(\ref{eqn:defining_eqn_3lines+cub_1st}) we obtain the Weierstrass model
\beq
\label{eqn:defining_fibration}
\begin{split}
 Y^2 & = X^3 - \big(t+\mu\big)\big(t+\nu\big) \Big( (c_1+2 d_2)t - (c_0 + d_1 + e_2)\Big) X^2\\
 &  + \big(c_1+d_2\big) \big(t+\mu\big)^2\big(t+\nu\big)^2 \Big(d_2 t^2 -(d_1 + 2 e_2) t + (d_0 +e_1)\Big) X \\
 & +\big(c_1+d_2\big)^2 \big(t+\mu\big)^3\big(t+\nu\big)^3 \Big( e_2 t^2 - e_1 t + e_0\Big) \,,
\end{split} 
\eeq 
with a discriminant function of the elliptic fibration $\Delta = (t+\mu)^6 (t+\nu)^6 (c_1+d_2)^2 p(t)$, where $p(t) = c_1^2 d_2^2 t^6 + \dots$ is a polynomial of degree six.  Given the Weierstrass model in Equation~(\ref{eqn:defining_fibration}) the statement is checked by explicit computation.
\end{proof}
Since we always assume $c_1 \neq0$ we have:
\begin{corollary}
\label{cor:limits1}
The  fibration in Lemma~\ref{lem:defining_eqn_3lines+cub} has the singular fibers $I_1^* + 2 I_0^* + 5 I_1$ if and only if $d_2=0$ and the remaining parameters are generic. It has the singular fibers $I_2^* + 2 I_0^* + 4 I_1$ if and only if $d_2=e_2=0$ and the remaining parameters are generic, and the singular fibers $I_3^* + 2 I_0^* + 3 I_1$ if and only if $d_2=e_2=e_1=0$ and the remaining parameters are generic. 
\end{corollary}
We also have the converse statement of Lemma~\ref{lem:defining_eqn_3lines+cub}:
\begin{proposition}
\label{lem:defining_eqn_3lines+cub_inverse}
A K3 surface admitting a Jacobian elliptic fibration with the singular fibers $3 I_0^* + 6 I_1$ and a trivial Mordell-Weil group arises as the double cover of the projective plane branched over three lines coincident in a point and a cubic.
\end{proposition}
\begin{proof}
Using a M\"obius transformation we can move the base points of the three singular fibers of type $I_0^*$ to $\mu, \nu, \infty$. An elliptic surface admitting the given Jacobian elliptic fibration then has a Weierstrass model of the form
\beq
\label{eqn:general_eqn}
 \begin{split}
 Y^2 & = X^3 + \big(t+\mu\big)\big(t+\nu\big) \Big( \tilde{c}_1 t + \tilde{c}_0\Big) X^2  +  \big(t+\mu\big)^2\big(t+\nu\big)^2 \Big( \tilde{d}_2 t^2 + \tilde{d}_1 t + \tilde{d}_0 \Big) X \\
 & +  \big(t+\mu\big)^3\big(t+\nu\big)^3 \Big(  \tilde{e}_3 t^3 + \tilde{e}_2 t^2 + \tilde{e}_1 t + \tilde{e}_0\Big) \,.
\end{split} 
\eeq 
A shift $X \mapsto X + \rho t (t+\mu)(t+\nu)$ eliminates the coefficient $\tilde{e}_3$ in Equation~(\ref{eqn:general_eqn}) if $\rho$ is a solution of $\rho^3 + \tilde{c}_1 \rho^2 + \tilde{d}_2 \rho + \tilde{e}_3 =0$. Thus, we can assume $\tilde{e}_3=0$. Next, let $c_1$ be a root of $c_1^2 = \tilde{c}_1^2 -4 \tilde{d}_2$. Then substituting
\beq
\begin{split}
c_0 = \frac{2 \tilde{d}_1}{c_1 - \tilde{c}_1} + \frac{4 \tilde{e}_2}{(c_1 - \tilde{c}_1)^2} + \tilde{c}_0, \quad
d_0 = \frac{2 \tilde{d}_0}{c_1 - \tilde{c}_1}  + \frac{4 \tilde{e}_1}{(c_1 - \tilde{c}_1)^2} , \quad
e_0=  \frac{4 \tilde{e}_0}{(c_1 - \tilde{c}_1)^2} , \\
d_1 =-\frac{2 \tilde{d}_2}{c_1 - \tilde{c}_1}  - \frac{8 \tilde{e}_1}{(c_1 - \tilde{c}_1)^2}, \quad
e_1=  -\frac{4 \tilde{e}_1}{(c_1 - \tilde{c}_1)^2} \,, \\
d_2 = - \frac{c_1+ \tilde{c}_1}{2}\,, \quad 
e_2=  \frac{4 \tilde{e}_2}{(c_1 - \tilde{c}_1)^2},
\end{split}
\eeq
into Equation~(\ref{eqn:defining_fibration}) recovers Equation~(\ref{eqn:general_eqn}).
\end{proof}
For the double $4\mathcal{H}$-surface $\mathcal{F} = \coprod_\xi \mathcal{F} _\xi$ in Proposition~\ref{prop:KSTTfibration} with the fibers $\mathcal{F} _\xi$ of genus one, we construct the relative Jacobian fibration $\coprod_\xi \mathrm{Jac}^0(\mathcal{F} _\xi)$.  We have the following:
\begin{theorem}
\label{thm:FpF}
The relative Jacobian fibration $\coprod_\xi \operatorname{Jac}(\mathcal{F} _\xi)$ associated with a generic double $4\mathcal{H}$-surface $\mathcal{F}$ is a K3 surface $\widetilde{\mathcal{Y}}$ obtained as the minimal resolution of the double-sextic surface for a generic configuration of three lines coincident in a point and a cubic. The latter defines an elliptic curve $\mathit{E}$ in $\mathbb{P}^2=\mathbb{P}(Z_1,  Z_2, Z_3)$ given by
\beq
\label{eqn:SWcurveII}
 \mathit{E}: \quad 0 = Z_1^3 + f(Z_2, Z_3) \, Z_1 + g(Z_2, Z_3) \,,
\eeq 
where $f, g$ were given by Equation~(\ref{eqn:hermite0b}).
\end{theorem}
\begin{proof}
It was shown in \cite{MR3995925} how a Weierstrass model for $\mathcal{F}'= \coprod_\xi \operatorname{Jac}(\mathcal{F} _\xi)$ is constructed explicitly. Applied to Equation~(\ref{eqn:g1fibration}) we obtain a Weierstrass model for $\mathcal{F}'$ given by
\beq
\begin{split}
Y^2 & = X^3 + \big(\xi + \mu \big)  \big(\xi + \nu \big) \, A_2 \, X^2 \\
& +  \big(\xi + \mu \big)^2  \big(\xi + \nu \big)^2 \, \Big( A_1 A_3 - 4 A_0 A_4\Big) \, X \\
& +  \big(\xi + \mu \big)^3  \big(\xi + \nu \big)^3 \, \Big( A_1^2 A_4 + A_0 A_3^2 - 4 A_0 A_2 A_4\Big) \,.
\end{split}
\eeq
This equation has the form of Equation~(\ref{eqn:general_eqn}) considered in Proposition~\ref{lem:defining_eqn_3lines+cub_inverse}.  For the Weierstrass model in Equation~(\ref{eqn:defining_fibration}) we then reconstruct the cubic in the branch locus by setting $Y=0$, rescaling $X \mapsto  \big(t+\mu\big)\big(t+\nu\big) X$, and extracting the irreducible cubic part. Using the defining equation for $\widetilde{\mathcal{Y}}$ in Table~\ref{tab:WEQ_2} yields Equation~(\ref{eqn:SWcurveII}).
\end{proof}
\section{Summary of results and discussion}
\label{sec:summary}
The highly non-trivial connection between families of K3 surfaces and their polarizing lattices appears in string theory as the manifestation of the F-theory/heterotic string duality. This viewpoint has been studied in \cites{MR3366121, MR3274790, MR3712162, MR3417046, MR3933163, MR4160930, MR3991815}.  We proved in Theorem~\ref{thm:duality} that there are algebraic correspondences between the K3 surfaces polarized by the rank-ten lattices $H \oplus N$ and $H\oplus E_8(-2)$. Since the moduli spaces $\mathfrak{M}_{H \oplus N}$ and  $\mathfrak{M}_{H \oplus E_8(-2)}$ are also the moduli spaces of F-theory models with discrete flux and the CHL string (heterotic string with CHL involution), respectively, Theorem~\ref{thm:duality} and Figure~\ref{fig:FTH-CHL} provide a mathematical framework for the duality between he CHL string in seven dimensions and the dual F-theory models.
\par  A natural 6-dimensional subspace that is contained simultaneously in both aforementioned moduli spaces is the subspace where the F-theory admits an additional anti-symplectic involution (induced by an involution on the base curve), and on the CHL string side one has an additional symplectic involution (namely, a van Geemen-Sarti involution). The duality diagram in Figure~\ref{fig:FTH-CHL} then extends to Figure~\ref{fig:FTH-CHL_2}. In Theorem~\ref{thm:subspace1} we proved an explicit parametrization for elements $\widetilde{\mathcal{F}}$ of  the moduli space of correspondences. A general element $\widetilde{\mathcal{F}}$ is a double-quadrics with a branch locus that is reducible and consists of two curves of bi-degree $(1,0)$ and $(0,1)$, respectively, and a genus-one curve in the linear system $|\mathcal{O}_{\mathbb{F}_0}(2,2)|$ such that the curves intersect in ordinary rational double points. The parametrization is then based on a construction of Andr\'e Weil \cite{MR717601}, in which the Abel-Jacobi map is used to obtain embeddings of genus-one curves as symmetric divisors of bi-degree $(2,2)$ in $\mathbb{F}_0 =\mathbb{P}^1 \times \mathbb{P}^1$. 
\par  We also showed that the F-theory moduli space has another natural 6-dimensional subspace, namely the moduli space of K3 surfaces polarized by the lattice $\langle 2 \rangle  \oplus \langle -2 \rangle \oplus D_4(-1)^{\oplus 3}$. This special situation corresponds to the case when, on the F-theory side, surfaces admit an additional symplectic involution and, on the CHL string side,  an additional anti-symplectic involution exists. The duality diagram in Figure~\ref{fig:FTH-CHL} then extends to Figure~\ref{fig:FTH-CHL_3}. The correspondences $\mathcal{F}$ turn out to then be precisely the double $4\mathcal{H}$-surfaces considered in \cite{MR1871336}. In Theorem~\ref{thm:FpF} we proved that the K3 surfaces $\widetilde{\mathcal{Y}}$ associated with the CHL string also carry a beautiful geometric description: they are special double-sextic surfaces  branched over a configuration of three distinct lines coincident in a point and an additional generic cubic. The latter divisor gives rise to an elliptic curve capturing part of the K3 moduli coordinates. 
\par  In the two special examples above, both involving 6-dimensional subspaces of $ \mathfrak{M}_{H \oplus N}$,  an elliptic curve naturally emerges.  This is not the elliptic curve upon which the CHL string is constructed. Rather, the elliptic curve underlying the heterotic string arises an an anti-canonical curve (cf.~\cite{MR1797021}), here, the anti-canonical curve in the rational elliptic surface $\mathcal{R}$ in Figure~\ref{fig:FTH-CHL_2} and Figure~\ref{fig:FTH-CHL_3}, respectively. The role of the elliptic curve underlying the heterotic string and the rational elliptic surface was investigated in previous work of the authors in \cite{MR3995925}. In contrast, the elliptic curve that emerges in the parametrization of the 6-dimensional subspaces above is a Seiberg-Witten type curve and parameterizes certain moduli of the F-theory/CHL vacua under consideration. The relation between this Seiberg-Witten type curve and the twisted principal $E_8 \times E_8$ bundle over the elliptic curve defining the heterotic string will be investigated in future work by the authors.
\bibliographystyle{amsplain}
\bibliography{references}{}
\end{document}